\documentclass[12pt]{amsart}

\usepackage{amssymb}

\usepackage{enumerate}

\makeatletter
\@namedef{subjclassname@2010}{%
  \textup{2010} Mathematics Subject Classification}
\makeatother

\usepackage[matrix,arrow,curve]{xy}

\newtheorem{thm}{Theorem}[section]
\newtheorem{cor}[thm]{Corollary}
\newtheorem{lem}[thm]{Lemma}
\newtheorem{lemma}[thm]{Lemma}
\newtheorem{prop}[thm]{Proposition}

\newtheorem{mthm}{Theorem}

\theoremstyle{definition}

\numberwithin{equation}{section}

\newcommand{\rad}{\operatorname{rad}}

\newcommand{\Ker}{\operatorname{Ker}}

\renewcommand{\mod}{\operatorname{mod}}
\newcommand{\umod}{\operatorname{\underline{mod}}}
\newcommand{\soc}{\operatorname{soc}}
\newcommand{\charact}{\operatorname{char}}

\newcommand{\Tr}{\operatorname{Tr}}
\newcommand{\op}{\operatorname{op}}

\newcommand{\bA}{\mathbb{A}}
\newcommand{\bD}{\mathbb{D}}
\newcommand{\bE}{\mathbb{E}}
\newcommand{\bL}{\mathbb{L}}

\newcommand{\cN}{\mathcal{N}}

\begin{document}

\baselineskip=17pt

\title[Socle deformed preprojective algebras]{Socle deformed preprojective algebras of generalized Dynkin type}

\author[J. Bia\l kowski]{Jerzy Bia\l kowski}
\address{Faculty of Mathematics and Computer Science\\
Nicolaus Copernicus University\\
Chopina 12/18, 87-100 Toru\'{n}, Poland}
\email{jb@mat.uni.torun.pl}

\date{}

\begin{abstract}
We provide a complete classification of 
finite-dimensional self-injective algebras 
which are socle equivalent to 
preprojective algebras of generalized Dynkin type.
In particular, we conclude that these algebras
are deformed preprojective algebras of generalized
Dynkin type (in the sense of \cite{BES1,ESk2}),
and hence are periodic algebras.
\end{abstract}



\maketitle

\section*{Introduction and the main results}
Throughout the article, $K$ will denote a fixed algebraically
closed field.
By an algebra we mean an associative, finite-dimensional $K$-algebra
with an identity, which we moreover assume to be basic and indecomposable.
Then any such an algebra $A$ can be written as a bound quiver algebra, that is,
$A \cong K Q / I$, where $Q = Q_A$ is the Gabriel quiver of $A$ and $I$ is
an admissible ideal in the path algebra $K Q$ of $Q$ \cite{ASS}.
For an algebra $A$, we denote by $\mod A$ the category of
finite-dimensional right $A$-modules and by $\Omega_A$ 
the syzygy operator % \emph{syzygy operator}
which assigns to a module $M$ in $\mod A$ the kernel $\Omega_A(M)$
of a minimal projective cover $P_A(M) \to M$ of $M$ in $\mod A$.
Then a module $M$ in $\mod A$ is called periodic if
$\Omega_A^n(M) \cong M$ for some $n \geq 1$,
and if so the minimal such $n$ is called the period of $M$.
Further, the category $\mod A$ is called   periodic %\emph{periodic}
if any
module $M$ in $\mod A$ without non-zero projective direct summands
is periodic.
It is known \cite{GSS} that the periodicity of $\mod A$ forces
$A$ to be   self-injective, % \emph{self-injective},
that is, the projective and injective modules in $\mod A$ coincide.
The category of finite-dimensional $A$-$A$-bimodules over an algebra $A$
is equivalent to the category $\mod A^e$ over the enveloping algebra
$A^e = A^{\op} \otimes_K A$ of $A$.
Then an algebra $A$ is called   periodic % \emph{periodic}
if $A$ is a periodic module in $\mod A^e$.
It is well known that if $A$ is a periodic algebra 
of period $n$ then for any indecomposable non-projective
module $M$ in $\mod A$ the syzygy 
$\Omega_A^n(M)$
is isomorphic to $M$ (see \cite{SY}).
It has been 
proved that all self-injective algebras
of finite representation type different from $K$, 
are periodic (see \cite{Du}).
Finding or possibly classifying periodic algebras 
is an important problem.
It is very interesting because of the connections 
with group theory, topology, singularity theory and cluster algebras.
For example, it was shown in \cite{ESk3,ESk4}
that the representation-infinite tame symmetric periodic algebras
of period $4$, with $2$-regular Gabriel quivers, 
are very specific deformations of the weighted surface algebras
of triangulated surfaces with arbitrarily oriented triangles.

A prominent class of periodic algebras is formed by the preprojective
algebras of generalized Dynkin type and their deformations.
Preprojective algebras were introduced by 
I.M. Gelfand and V.A. Ponomarev \cite{GP}
(and implicitely in the work of C.  Riedtmann \cite{Rd}) to study the
preprojective representations of finite quivers, 
and occurred in very different contexts.
The finite-dimensional preprojective algebras are exactly the
preprojective algebras $P(\Delta)$ associated to the   generalized Dynkin graphs % \emph{generalized Dynkin graphs}
$\mathbb{A}_n (n \geq 1)$, $\mathbb{D}_n (n \geq 4)$,
$\mathbb{E}_6$, $\mathbb{E}_7$, $\mathbb{E}_8$, $\mathbb{L}_n (n \geq 1)$.
It follows from 
%\cite{HPR1}, \cite{HPR2}
\cite{HPR1,HPR2} 
that these are exactly the
graphs associated to the indecomposable finite symmetric Cartan
matrices having subadditive functions which are not additive.
We also mention that the preprojective algebras $P(\Delta)$ of Dynkin
types $\Delta \in \{ \mathbb{A}_n, \mathbb{D}_n, \mathbb{E}_6,
\mathbb{E}_7, \mathbb{E}_8 \}$ are the stable Auslander algebras
of the categories of maximal Cohen-Macaulay modules
of the Kleinian 2-dimensional hypersurface singularities
$K[[x,y,z]]/(f_{\Delta})$ 
%(see \cite{AR1}, \cite{AR2}, \cite{ESk2}).
(see \cite{AR1,AR2,ESk2}).
Moreover, for each $n \geq 1$, the preprojective algebra
$P(\mathbb{L}_n)$ is the stable Auslander algebra of the category of
maximal Cohen-Macaulay modules over the simple plane curve singularity
$K[[x,y]]/(x^{2n+1} + y^2)$ 
%(see \cite{DW}, \cite{ESk2}).
(see \cite{DW,ESk2}).
The preprojective algebras of Dynkin types have been  applied
by C.~Geiss, B.~Leclerc and J.~Schr\"oer to study the structure of cluster
algebras related to semisimple and unipotent
algebraic groups (see \cite{GLS}).
We also note that the preprojective algebras
of generalized Dynkin type are periodic of period
dividing $6$ 
%(see \cite{AR2}, \cite{BBK}, \cite{BES2}, \cite{ESn}).
(see \cite{AR2,BBK,BES2,ESn}).

In the paper we are concerned with the classification
of isomorphism classes 
of the deformations of preprojective algebras
of generalized Dynkin type introduced in \cite{BES1}.
Namely, to each generalized Dynkin graph $\Delta$ one associates
a finite-dimensional 
local self-injective $K$-algebra
$R(\Delta)$ and a deformed preprojective algebra of type $\Delta$ is the
deformation $P^f(\Delta)$ of $P(\Delta)$ given by an admissible element
$f$ of the radical square of $R(\Delta)$, and $P^f(\Delta) = P(\Delta)$
for $f = 0$ (see Section~\ref{sec:1} for details).
It has been proved in \cite{BES1} that the deformed preprojective algebras
of generalized Dynkin types are (finite-dimensional) periodic
algebras and form the representatives 
of the isomorphism classes of all 
indecomposable self-injective algebras $A$ for
which the third syzygy $\Omega_A^3(S)$ of any non-projective simple
module $S$ in $\mod A$ is isomorphic 
to its Nakayama shift $\cN_A(S)$.
Then it follows that every 
self-injective algebra $A$ with the 
stable module category $\umod A$ $2$-Calabi-Yau 
is isomorphic to
a deformed preprojective algebra $P^f(\Delta)$ of a generalized Dynkin
type $\Delta$, and it is an interesting open problem when the converse
is true.
Therefore, it is important to have a complete classification
of the isomorphism classes of deformed preprojective
algebras of generalized Dynkin type.
The isomorphism classes of the deformed preprojective  algebras of types
$\mathbb{L}_n$, $n \geq 1$, were classified in \cite{BES2}.
Moreover, it was shown in \cite{BES2} that these algebras
are exactly the stable Auslander algebras of simple curve
singularities of Dynkin type $\mathbb{A}_{2n}$.
Furthermore, A.~Dugas proved in \cite{Du2} that the stable
Auslander algebras of arbitrary hypersurface singularities
of finite Cohen-Macaulay type are periodic algebras.
Hence it would be interesting to understand connections
of these algebras with deformed preprojective algebras
of Dynkin type as well as describe their quiver presentations.
In \cite{B}
we proved that every deformed preprojective algebra
of type $\bE_6$ is isomorphic to the preprojective
algebra $P(\bE_6)$.
On the other hand, 
a classification of the isomorphism classes of all deformed
preprojective algebras of Dynkin types $\mathbb{D}_n (n \geq 4)$,
$\mathbb{E}_7$, $\mathbb{E}_8$ 
seems to be currently a hard problem.

Recall that by a classical  Nakayama's result 
%(see \cite{Na}, \cite{SY})
(see \cite{Na,SY})
the left socle $\soc({}_A A)$ and the right socle $\soc(A_A)$
of a self-injective algebra $A$ coincide, and hence
$\soc({}_A A) = \soc(A_A)$
is a two-sided ideal of $A$, called the socle of $A$
and denoted by $\soc(A)$.
In the paper, two self-injective algebras $A$ and $B$ are said
to be \emph{socle equivalent} if the quotient algebras
$A/\soc(A)$ and $B/\soc(B)$ are isomorphic.
Moreover, by a \emph{socle deformed preprojective algebra of
generalized Dynkin type $\Delta$} 
is meant
a deformed preprojective
algebra $P^f(\Delta)$ of type $\Delta$ 
which is socle equivalent but non-isomorphic to $P(\Delta)$.

In Section~\ref{sec:preliminaries} we introduce 
deformed preprojective algebras 
  $P^{*}(\bD_{2m})$, $m \geq 2$,
  $P^{*}(\bE_7)$,
  $P^{*}(\bE_8)$, 
  $P^{*}(\bL_n)$, $n \geq 2$.
Then the first main result of the paper is as follows.

\begin{mthm}
\label{thm1}
Let $\Lambda$ be a basic, indecomposable, finite-dimensional
self-injective algebra over an algebraically closed field $K$.
Then $\Lambda$ is socle equivalent but not isomorphic
to a preprojective algebra
of generalized Dynkin type if and only if $K$ is of
characteristic $2$ and $\Lambda$ is isomorphic to one
of the algebras
  $P^{*}(\bD_{2m})$, $m \geq 2$,
  $P^{*}(\bE_7)$,
  $P^{*}(\bE_8)$, 
  $P^{*}(\bL_n)$, $n \geq 2$.
\end{mthm}

The second main result describes symmetricity  properties of
socle deformed preprojective algebras
of generalized Dynkin type.

\begin{mthm}
\label{thm3}
Let $\Lambda$ be a socle deformed preprojective algebra
of generalized Dynkin type. 
Then the following hold.
\begin{enumerate}[(i)]
 \item
  $\Lambda$ is a weakly symmetric algebra.
 \item
  $\Lambda$ is a symmetric algebra if and only if $\Lambda$
  is isomorphic to the algebra $P^{*}(\bL_n)$ for some $n \geq 2$.
\end{enumerate}
\end{mthm}

We obtain also the following consequence of
Theorem~\ref{thm1}.

\begin{mthm}
\label{thm4}
Let $\Delta$ be a generalized Dynkin type different from $\bA_1$ and $\bL_1$,
and $K$ an algebraically closed field.
Then the following conditions are equivalent:
\begin{enumerate}[(i)]
 \item
  There exists a socle deformed 
  preprojective algebra $P^f(\Delta)$ of type $\Delta$ over $K$.
 \item
  The preprojective algebra $P(\Delta)$ of type $\Delta$ over $K$
  is a symmetric algebra,
  and 
  $K$ is of characteristic $2$,
  if $\Delta$ is of the type $\bL_n$.
 \item
  $\Delta$ is one of the types
  $\bD_{2m}$, $m \geq 2$,
  $\bE_7$,
  $\bE_8$,
  $\bL_n$, $n \geq 2$,
  and $K$ is of characteristic $2$.
\end{enumerate}
\end{mthm}

The paper is organized as follows.
In Section~\ref{sec:preliminaries} we present some
definitions and results
on deformed preprojective algebras,
essential for further considerations.
Section~\ref{sec:Ln} 
is devoted to known results on the deformed preprojective algebras
of type $\bL_n$, $n \geq 1$.
In Section~\ref{sec:no-def} 
we prove that 
each algebra $A$ socle equivalent to the algebra  
$P(\Delta)$, for 
$\Delta \in \{ \mathbb{A}_n, \mathbb{D}_{2m+1}, \mathbb{E}_6 \}$,
$n \geq 1$, $m \geq 2$,
is in fact isomorphic to $P(\Delta)$
(in any characteristic).
In Section~\ref{sec:only-canoninical} 
we consider 
algebras socle equivalent to the algebras  
$P(\Delta)$, for
$\Delta \in \{ \mathbb{D}_{2m}, \mathbb{E}_7, \mathbb{E}_8, \mathbb{L}_{n} \}$, $m \geq 2$, $n \geq 2$,
and 
prove that each of them is isomorphic 
to the preprojective algebra $P(\Delta)$
or to the algebra $P^*(\Delta)$.
Section~\ref{sec:char-2} 
is devoted to the proof
that 
the algebras socle equivalent but not isomorphic
to preprojective algebras
of generalized Dynkin type
exist only in characteristic $2$.
In Section~\ref{sec:non-sym} 
we prove that
in characteristic $2$
the algebras 
$P^{*}(\bD_{2m})$, $m \geq 2$,
$P^{*}(\bE_7)$,
$P^{*}(\bE_8)$
are 
weakly symmetric but
not symmetric,
and derive some consequences.
In the final 
Section~\ref{sec:last} 
we combine the
results of the previous sections and
complete the proofs of the main results.

\smallskip

For general background on the relevant 
representation theory we refer to the books
\cite{ASS,SY}
and the survey article \cite{ESk2}.

\section{Deformed preprojective algebras of generalized Dynkin type}
\label{sec:1}
\label{sec:preliminaries}

For a generalized Dynkin graph $\Delta$, we consider the associated
(double) quiver $Q_{\Delta}$ as follows:
\[
    \parbox[t]{2.3cm}{$\Delta = \mathbb{A}_n :\\ $ \ $(n \geq 1)$}
    \xymatrix{
        0 \ar@<.5ex>^{a_0}[r] & 1 \ar@<.5ex>^{\bar{a}_0}[l] \ar@<.5ex>^{a_1}[r] & 2
        \ar@<.5ex>^{\bar{a}_1}[l] \ar@{*{ }*{\,.\,}*{ }}[r] & n-2 \ar@<.5ex>^{a_{n-2}}[r] &
         n - 1 \ar@<.5ex>^(.4){\bar{a}_{n-2}}[l] \\
    }
\]
\[
    \parbox[t]{2.3cm}{$\Delta = \mathbb{D}_n :\\ $ \ $(n \geq 4)$}
    \xymatrix{
        0 \ar@<.5ex>^{a_0}[rd] \\
         & 2 \ar@<.5ex>^{\bar{a}_0}[lu] \ar@<.5ex>^{\bar{a}_1}[ld] \ar@<.5ex>^{a_2}[r] & 3
        \ar@<.5ex>^{\bar{a}_2}[l] \ar@{*{ }*{\,.\,}*{ }}[r] & n-2 \ar@<.5ex>^{a_{n-2}}[r] &
         n - 1 \ar@<.5ex>^(.4){\bar{a}_{n-2}}[l] \\
        1 \ar@<.5ex>^{a_1}[ru] \\
    }
\]
\[
    \parbox[t]{2.3cm}{$\Delta = \mathbb{E}_n : \\ $ \ $(6 \leq n \leq 8)$}
    \xymatrix{
        &&0 \ar@<.5ex>^{a_0}[d] \\
        1 \ar@<.5ex>^{a_1}[r] & 2 \ar@<.5ex>^{\bar{a}_1}[l] \ar@<.5ex>^{a_2}[r] &
        3 \ar@<.5ex>^{\bar{a}_2}[l] \ar@<.5ex>^{a_3}[r] \ar@<.5ex>^{\bar{a}_0}[u] &
        4
        \ar@<.5ex>^{\bar{a}_3}[l] \ar@{*{ }*{\,.\,}*{ }}[r] & n-2 \ar@<.5ex>^{a_{n-2}}[r] &
         n-1 \ar@<.5ex>^(.4){\bar{a}_{n-2}}[l] \\
    }
\]
\[
    \parbox[c]{2.3cm}{$\Delta = \mathbb{L}_n :\\ $ \ $(n \geq 1)$}
    \xymatrix{
        0
        \ar `ld_u[] `_rd[]^{\varepsilon = \bar{\varepsilon}} []
        \ar@<.5ex>^{a_0}[r] & 1 \ar@<.5ex>^{\bar{a}_0}[l] \ar@<.5ex>^{a_1}[r] & 2
        \ar@<.5ex>^{\bar{a}_1}[l] \ar@{*{ }*{\,.\,}*{ }}[r] & n-2 \ar@<.5ex>^{a_{n-2}}[r] &
         n - 1 \ar@<.5ex>^(.4){\bar{a}_{n-2}}[l] \\
    } .
\]
Therefore, the graph $\Delta$ is obtained from the quiver $Q_{\Delta}$ by
replacing each pair $\{a,\bar{a}\}$ of arrows by an indirect edge.
The preprojective algebra $P(\Delta)$ of type $\Delta$ is then
the bound quiver algebra $K Q_{\Delta} / I_{\Delta}$,
where $I_{\Delta}$ is the ideal of the path algebra
$K Q_{\Delta}$ of $Q_{\Delta}$ generated by the elements
\[
    \sum\limits_{a,ia =v} a \bar{a}, 
\]
for all vertices $v$ of $Q_{\Delta}$, where $\bar{\bar{a}} = a$ and $i a$ is the starting
vertex of an arrow $a$.
Then $P(\Delta)$ is a finite-dimensional self-injective algebra.
Moreover, for $\Delta \in \{ \mathbb{A}_n, \mathbb{D}_n,
\mathbb{E}_6, \mathbb{E}_7, \mathbb{E}_8, \mathbb{L}_n \} \setminus \{ \mathbb{A}_1 \}$
we have $P(\Delta)$ is non-simple periodic with
$\Omega_{P(\Delta)^e}^6(P(\Delta)) \cong P(\Delta)$.

Let $\Delta$ be a generalized Dynkin graph.
Following \cite{BES1} we associate to
$\Lambda$ the $K$-algebra $R(\Delta)$ as follows
\[
 \begin{array}{ll}
 R(\mathbb{A}_n) = K; & n \geq 1; \\
 R(\mathbb{D}_n) = K \langle x,y \rangle / (x^2, y^2, (x+y)^{n-2}); & n \geq 4; \\
 R(\mathbb{E}_n) = K \langle x,y \rangle / (x^2, y^3, (x+y)^{n-3}); & 6 \leq n \leq 8; \\
 R(\mathbb{L}_n) = K [x] / (x^{2n}); & n \geq 1. \\
 \end{array}
\]

Further, we choose the exceptional vertex in the Gabriel quiver
$Q_{P(\Delta)}$ of 
$P(\Delta)$ as: $0, 2, 3, 3, 3$ and $0$ if
$\Delta = \mathbb{A}_n, \mathbb{D}_n, \mathbb{E}_6, \mathbb{E}_7,
\mathbb{E}_8$ and $\mathbb{L}_n$, respectively. Moreover, we
denote by $e_{\Delta}$ the primitive idempotent of $P(\Delta)$ associated
to the chosen exceptional vertex of $Q_{P(\Delta)}$. Then a simple
checking shows that $R(\Delta)$ is isomorphic to $e_{\Delta} P(\Delta) e_{\Delta}$,
and hence is a local, finite-dimensional and self-injective $K$-algebra.

For $\Delta \in \{ \mathbb{A}_n, \mathbb{D}_n, \mathbb{L}_n\}$, every
element $f$ from the square $\rad^2 R(\Delta)$ of the radical
of $R(\Delta)$ is said to be admissible. %\textit{admissible}.

If $\Delta = \mathbb{E}_n (6 \leq n \leq 8)$ an element
$\rad^2 R(\Delta)$ is said to be admissible if it satisfies the following
condition:
$$
   \big (x+y+f(x,y)\big)^{n-3} = 0 .
$$

For an admissible element $f \in R(\Delta)$, the
\emph{deformed preprojective algebra}
$P^f(\Delta)$ of type $\Delta$,   with
respect to $f$, is defined to be as the bound quiver algebra $K Q_{\Delta}/I_{\Delta}^f$,
where $I_{\Delta}^f$ is the ideal in the
path algebra $K Q_{\Delta}$ of $_{\Delta}Q$ generated by the elements
\[
  \sum_{a , i a = v} a \bar{a},
\]
for all ordinary vertices $v$ of $Q_{\Delta}$, 
and
\[
 \begin{array}{l@{\mbox{\ \ if \,}}l}
  a_0 \bar{a}_0, &  \Delta = \mathbb{A}_n ;\\
  \bar{a}_0 a_0 + \bar{a}_1 a_1 + a_2 \bar{a}_2 +
   f(\bar{a}_0 a_0, \bar{a}_1 a_1), \
   (\bar{a}_0 a_0 + \bar{a}_1 a_1)^{n-2}, &
 \Delta = \mathbb{D}_n ;\\
  \bar{a}_0 a_0 + \bar{a}_2 a_2 + a_3 \bar{a}_3 +
   f(\bar{a}_0 a_0, \bar{a}_2 a_2), \
   (\bar{a}_0a_0 + \bar{a}_2a_2)^{n-3}, &
   \Delta = \mathbb{E}_n  %\  (n=6, 7, 8)
   ; \\
  \varepsilon^2 + a_0 \bar{a}_0 + \varepsilon f(\varepsilon), \
    \varepsilon^{2n},  &  \Delta = \mathbb{L}_n . \\
 \end{array}
\]

Clearly, $P(\mathbb{A}_n)$ is the unique deformed preprojective
algebra of type $\mathbb{A}_n$, since $\rad R(\mathbb{A}_n) = 0$.
We note that the deformed preprojective algebras of generalized
Dynkin type are, with the exception of few small cases, of
wild representation type (see \cite{ESk1}).
It has been proved in \cite{BES1} that any deformed preprojective
algebra $P^f(\Delta)$ of generalized Dynkin type $\Delta$ is a periodic
algebra but the proof presented there does not allow to determine
its period.

We introduce now
canonical 
deformed preprojective algebras of generalized Dynkin type
different from $\mathbb{A}_n$ and $\mathbb{L}_1$,
which are socle equivalent to preprojective algebras 
of generalized Dynkin type.
We denote:
\begin{itemize}
 \item
    $P^{*}(\mathbb{D}_{2m}) = P^{f}(\mathbb{D}_{2m})$, for $f$ the coset of $(x y)^{m-1}$ in $R(\mathbb{D}_{2m})$, $m \geq 2$;
 \item
    $P^{*}(\mathbb{D}_{2m+1}) = P^{f}(\mathbb{D}_{2m+1})$, for $f$ the coset of $(x y)^{m-1}x$ in $R(\mathbb{D}_{2m+1})$, $m \geq 2$;
 \item
    $P^{*}(\mathbb{E}_6) = P^{f}(\mathbb{E}_6)$, for $f$ the coset of $(yx)^{2}y$ in $R(\mathbb{E}_6)$;
 \item
    $P^{*}(\mathbb{E}_n) = P^{f}(\mathbb{E}_n)$, for $f$ the coset of $(xy)^{3n-17}$ in $R(\mathbb{E}_6)$, $7 \leq n \leq 8$;
 \item
    $P^{*}(\mathbb{L}_n) = P^{f}(\mathbb{L}_n)$, for $f$ the coset of $x^{2n-2}$ in $R(\mathbb{L}_n)$, $n \geq 2$.
\end{itemize}

\section{Deformed preprojective algebras of generalized Dynkin type $\bL_n$}
\label{sec:Ln}

In this section we present known facts on the structure of
deformed preprojective algebras of generalized Dynkin type $\bL_n$.
We recall that these algebras were studied by
J.~Bia\l kowski, K.~Erdmann and A.~Skowro\'nski \cite{BES1,BES2},
K.~Erdmann and A.~Skowro\'nski \cite{ESk2},
and
T.~Holm and A.~Zimmermann \cite{HZ}.

Following the notation from \cite{BES2}
we distinguish special deformed preprojective algebras of type $\bL_n$:
\[
    L_n^{(r)} = P^{f_r}(\bL_n)
    \mbox{ with }
    f_r = x^{2 r} + (x^{2 n}),
    \mbox{ for }
    r \in \{1,\dots,n\}.
\]
We note, that in the above notation, 
$L_n^{(n)} = P^{f_n}(\bL_n)$ is the ordinary
{preprojective algebra} $P(\bL_n)$ of type $\bL_n$
and $L_n^{(n-1)} = P^{f_{n-1}}(\bL_n) = P^{*}(\bL_n)$
is its socle deformation.
Moreover, we have the following fact \cite[Proposition~6.1]{BES1}.

\begin{prop}
\label{prop:Ln-clas}
If $K$ is of characteristic $2$, 
then the algebras $L_n^{(1)},\dots,L_n^{(n)}$
form a family of pairwise non-isomorphic
deformed preprojective algebras of generalized Dynkin type $\bL_n$.
\end{prop}

We have
also
the following classification
of deformed preprojective algebras of
type $\bL_n$, 
formulated in \cite[Theorem~2]{BES2}.

\begin{thm}
\label{thm:Ln-clas}
Let $\Lambda = P^f(\bL_n)$ be a deformed preprojective algebra
of type $\bL_n$ over an algebraically closed field $K$.
Then the following statements hold.
\begin{enumerate}[(i)]
 \item
  If $K$ is of characteristic different from $2$, then
  $\Lambda$ is isomorphic to the preprojective algebra $P(\bL_n)$.
 \item
  If $K$ is of characteristic $2$, then $\Lambda$ is isomorphic to
  an algebra $L_n^{(r)}$, for some $r \in \{1,\dots,n\}$.
\end{enumerate}
\end{thm}

We have also the following fact from \cite[Corollary~4]{BES2}.

\begin{prop}
\label{thm:Ln-sym}
Let $\Lambda = P^f(\bL_n)$ be a deformed preprojective algebra
of type $\bL_n$. Then $\Lambda$ is a symmetric algebra.
\end{prop}

Therefore, it follows that both 
the preprojective algebra $P(\bL_n) = P(\bL_n^{(n)})$ 
of generalized Dynkin type $\bL_n$ 
and 
the socle deformed preprojective algebra 
$P^*(\bL_n) = P(\bL_n^{(n-1)})$ of generalized Dynkin type $\bL_n$   
are symmetric algebras.

\section{Algebras without proper socle deformations}
%\label{sec:2}
\label{sec:no-def}

The aim of this section is to prove the following 
proposition.

\begin{prop}
\label{prop:no-def}
Let $\Lambda$ be a preprojective algebra
of Dynkin type 
$\Delta \in \{ \mathbb{A}_n, \mathbb{D}_{2m+1}, \mathbb{E}_6 \}$,
$n \geq 1$, $m \geq 2$,
and let $A$ be an algebra socle equivalent to $\Lambda$.
Then $A$ and $\Lambda$ are isomorphic.
\end{prop}

We divide the proof of this proposition into 
three cases.

\begin{lemma}
Let $\Lambda$ be a preprojective algebra
of Dynkin type $\mathbb{A}_n$, $n \geq 1$, 
and let $A$ be a self-injective algebra socle equivalent to $\Lambda$.
Then the algebras $A$ and $\Lambda$ are isomorphic.
\end{lemma}

\begin{proof}
If $n$ is an even positive integer, then 
for each maximal non-zero path $\omega$ 
in $P(\mathbb{A}_n)$
its source $s(\omega)$ is different from target $t(\omega)$,
because $s(\omega) + t(\omega) = n-1$.
Therefore $P(\mathbb{A}_n)$ has no non-trivial socle deformation,
and so $A$ and $\Lambda$ are isomorphic.

Hence assume that $n = 2m+1$ for some non-negative integer $m$
and let $A$ be a self-injective algebra socle equivalent to the 
preprojective algebra $P(\mathbb{A}_n)$
of generalized Dynkin type $\mathbb{A}_n$.
Then $A$ is isomorphic to a 
bound quiver algebra $A'$ of the
path algebra of the quiver
\[
    \xymatrix@C=1pc{
        0 \ar@<.5ex>^{a_0}[rr] && 1
        \ar@<.5ex>^{\bar{a}_0}[ll]
         \ar@{-}@<.5ex>[r] & \ar@<.5ex>[l] \dots \ar@<.5ex>[r] & \ar@{-}@<.5ex>[l]
         n-2 \ar@<.5ex>^{a_{n-2}}[rr] &&
         n - 1 \ar@<.5ex>^(.5){\bar{a}_{n-2}}[ll] \\
    }
\]
bound by relations of the form
\begin{gather*}
  a_0 \bar{a}_0 = 0 
  , \ \
  \bar{a}_{n-1} a_{n-1} = 0
 ,\\
  \bar{a}_{m-1} a_{m-1} +  a_{m} \bar{a}_{m}
  +
  \theta \bar{a}_{m-1} \dots \bar{a}_1 \bar{a}_0 a_0 a_1 \dots {a}_{m-1}
   = 0
 ,\\
  \bar{a}_{l-1} a_{l-1} +  a_{l} \bar{a}_{l} = 0
  , \mbox{ for } l \in \{ 0, \dots, m-1, m+1 \dots, n-2 \}
 ,\\
   \bar{a}_{m-1} \dots \bar{a}_1 \bar{a}_0 a_0 a_1 \dots {a}_{m-1} {a}_m = 0
  , \ \
   \bar{a}_m \bar{a}_{m-1} \dots \bar{a}_1 \bar{a}_0 a_0 a_1 \dots {a}_{m-1} = 0
  ,
\end{gather*}
for some coefficient $\theta \in K$.
Observe moreover that the following equalities hold in the algebras
$P(\mathbb{A}_{n})$ and $A'$
\begin{align*}
  \bar{a}_{m-2} \dots \bar{a}_0 a_0 \dots {a}_{m-1} \bar{a}_{m-1}
  &=
   - \bar{a}_{m-2} \dots \bar{a}_0 a_0 \dots {a}_{m-2} \bar{a}_{m-2} {a}_{m-2}
  \\&= \dots
  =
   (-1)^{m-1} \bar{a}_{m-2} \dots \bar{a}_1 \bar{a}_0 a_0 \bar{a}_0 a_0 a_1 \dots {a}_{m-2}
  \\&= 0
  .
\end{align*}
We will show that algebras $P(\mathbb{A}_{n})$ and $A'$ are isomorphic.
Let $\varphi : P(\mathbb{A}_n) \to A'$ 
be the homomorphism of algebras defined on arrows as follows
\begin{align*}
%\begin{gather*}
  \varphi(a_{m-1}) &= a_{m-1}
    +
  \theta \bar{a}_{m-2} \dots \bar{a}_1 \bar{a}_0 a_0 a_1 \dots {a}_{m-1}
  ,\\
  \varphi(a_k) &= a_k
  , \mbox{ for }  k \in \{ 0,\dots,m-2,m,\dots,n-2 \}
  ,\\
  \varphi(\bar{a}_l) &= \bar{a}_l
  , \mbox{ for } l \in \{ 0,\dots,n-1 \}
  .
%\end{gather*}
\end{align*}
We note that $\varphi$ is well defined.
Indeed, we have the equalities
\begin{align*}
    \varphi(a_0 \bar{a}_0)
    &= \varphi(a_0) \varphi(\bar{a}_0)
     = a_0 \bar{a}_0
     = 0,
%\end{align*}\begin{align*}
    \\
    \varphi(\bar{a}_{m-2} a_{m-2} + a_{m-1} \bar{a}_{m-1}) &=
    \varphi(\bar{a}_{m-2}) \varphi(a_{m-2}) + \varphi(a_{m-1}) \varphi(\bar{a}_{m-1})
    \\ &=
    \bar{a}_{m-2} a_{m-2} + a_{m-1} \bar{a}_{m-1}
     \\ & \ \
         + \theta \bar{a}_{m-2} \dots \bar{a}_1 \bar{a}_0 a_0 a_1 \dots {a}_{m-1} \bar{a}_{m-1}
         = 0
         ,
\end{align*}\begin{align*}
%    \\
    \varphi(\bar{a}_{m-1} a_{m-1} + a_{m} \bar{a}_{m}) &=
    \varphi(\bar{a}_{m-1}) \varphi(a_{m-1}) + \varphi(a_{m}) \varphi(\bar{a}_{m})
    \\ &=
    \bar{a}_{m-1} a_{m-1} + a_{m} \bar{a}_{m}
     \\ & \ \
         + \theta \bar{a}_{m-1} \bar{a}_{m-2} \dots \bar{a}_1 \bar{a}_0 a_0 a_1 \dots {a}_{m-1}
         = 0
         ,
    \\
    \varphi(\bar{a}_{k-1} a_{k-1} + a_{k} \bar{a}_{k}) &=
    \varphi(\bar{a}_{k-1}) \varphi(a_{k-1}) + \varphi(a_{k}) \varphi(\bar{a}_{k})
%    \\ &=
     =
    \bar{a}_{k-1} a_{k-1} + a_{k} \bar{a}_{k}
     \\&= 0
   , \mbox{ for } k \in \{ 1,\dots, m-2,m+1,\dots,n-2 \}
         ,
    \\
    \varphi(\bar{a}_{n-2} a_{n-2}) &=
       \varphi(\bar{a}_{n-2}) \varphi(a_{n-2}))
       =
       \bar{a}_{n-2} a_{n-2}
       = 0
         ,
\end{align*}
so $\varphi$  is well defined. 
Similarly, we define the homomorphism
$\psi : A' \to P(\mathbb{A}_n)$ 
of algebras given on the arrows as follows:
\begin{align*}
%\begin{gather*}
  \psi(a_{m-1}) &= a_{m-1}
    -
  \theta \bar{a}_{m-2} \dots \bar{a}_1 \bar{a}_0 a_0 a_1 \dots {a}_{m-1}
  ,\\
  \psi(a_k) &= a_k
  , \mbox{ for }  k \in \{ 0,\dots,m-2,m,\dots,n-2 \}
  ,\\
  \psi(\bar{a}_l) &= \bar{a}_l
  , \mbox{ for } l \in \{ 0,\dots,n-1 \}
  ,
%\end{gather*}
\end{align*}
and prove that it is well defined.
The compositions $\psi \varphi$ and $\varphi \psi$ are 
the identities on $P(\mathbb{A}_n)$ and $A'$, respectively.
Hence $\psi$ and $\varphi$ are isomorphisms,
and the algebras $P(\mathbb{A}_n)$ and $A'$ are isomorphic.
Therefore, the algebras $P(\mathbb{A}_n)$ and $A$ are also isomorphic.
\end{proof}

\begin{lemma}
Let $A$ be a self-injective algebra socle equivalent to the 
preprojective algebra $P(\mathbb{E}_6)$.
Then the algebras $A$ and $P(\mathbb{E}_6)$ are isomorphic.
\end{lemma}

\begin{proof}
The algebra $A$ is isomorphic to a bound quiver algebra $A'$ 
of the quiver $Q_{\mathbb{E}_6}$
\[
    \xymatrix{
        &&0 \ar@<.5ex>^{a_0}[d] \\
        1 \ar@<.5ex>^{a_1}[r] & 2 \ar@<.5ex>^{\bar{a}_1}[l] \ar@<.5ex>^{a_2}[r] &
        3 \ar@<.5ex>^{\bar{a}_2}[l] \ar@<.5ex>^{a_3}[r] \ar@<.5ex>^{\bar{a}_0}[u] &
        4 \ar@<.5ex>^{\bar{a}_3}[l] \ar@<.5ex>^{a_4}[r] &
        5 \ar@<.5ex>^(.5){\bar{a}_4}[l] \\
    }
\]
bound by the relations:
\begin{gather*}
  a_0 \bar{a}_0 = \theta_0 a_0 \bar{a}_2 a_2 \bar{a}_0 a_0 (\bar{a}_2 a_2)^2 \bar{a}_0
  , \ \
  a_1 \bar{a}_1 = 0
  , \ \
  \bar{a}_1 a_1 + a_2 \bar{a}_2 = 0
 ,\\
  \bar{a}_0 a_0 + \bar{a}_2 a_2 + a_3 \bar{a}_3 = \theta_3 (\bar{a}_2 a_2)^2 \bar{a}_0 a_0 (\bar{a}_2 a_2)^2
  , \ \
  \bar{a}_3 a_3 + a_4 \bar{a}_4 = 0
  , \ \
  \bar{a}_4 a_4 = 0
 ,\\
  a_0 \bar{a}_0 a_0 = 0
  , \ \
  \bar{a}_0 a_0 \bar{a}_0 = 0
  , \ \
  a_0 (\bar{a}_2 a_2 + a_3 \bar{a}_3) = 0
  , \ \
  (\bar{a}_2 a_2 + a_3 \bar{a}_3) \bar{a}_0 = 0
 ,
 \\
%\end{gather*}\begin{gather*}
  a_2 (\bar{a}_0 a_0 + \bar{a}_2 a_2 + a_3 \bar{a}_3) = 0
  , \ \
  (\bar{a}_0 a_0 + \bar{a}_2 a_2 + a_3 \bar{a}_3) \bar{a}_2 = 0
 ,\\
  \bar{a}_3 (\bar{a}_0 a_0 + \bar{a}_2 a_2 + a_3 \bar{a}_3) = 0
  , \ \
  (\bar{a}_0 a_0 + \bar{a}_2 a_2 + a_3 \bar{a}_3) a_3 = 0
  ,
\end{gather*}
for some coefficients $\theta_0, \theta_3 \in K$.
We will show that the algebras $P(\mathbb{E}_6)$ and $A'$ are isomorphic.
Observe first that, in these algebras, we have the equalities
%\begin{align}
\begin{gather}
\label{eq:1}
\tag{1}
 (\bar{a}_0 a_0)^2 = \bar{a}_0 (a_0 \bar{a}_0) a_0 = 0, \\
\label{eq:2}
\tag{2}
 (\bar{a}_2 a_2)^3 = \bar{a}_2 (a_2 \bar{a}_2)^2 a_2
     = - \bar{a}_2 (\bar{a}_1 a_1)^2 a_2 = 0,
%\end{align}
\end{gather}
and
\begin{align}
 0 &= a_3 (a_4 \bar{a}_4)^2 \bar{a}_3
    = - (a_3 \bar{a}_3)^3
    = (\bar{a}_0 a_0 + \bar{a}_2 a_2)^3
   \nonumber \\
\label{eq:3}
\tag{3}
  &=\bar{a}_0 a_0 \bar{a}_2 a_2 \bar{a}_0 a_0
  + \bar{a}_2 a_2 \bar{a}_0 a_0 \bar{a}_2 a_2
  + (\bar{a}_2 a_2)^2 \bar{a}_0 a_0
  + \bar{a}_0 a_0 (\bar{a}_2 a_2)^2 .
\end{align}
Then from (\ref{eq:2}) and (\ref{eq:3}) we have the equality
\begin{gather}
\label{eq:4}
\tag{4}
 (\bar{a}_2 a_2)^2 \bar{a}_0 a_0 (\bar{a}_2 a_2)^2
   +
 \bar{a}_0 a_0 \bar{a}_2 a_2 \bar{a}_0 a_0 (\bar{a}_2 a_2)^2
   = 0 .
\end{gather}
Let $\varphi : P(\mathbb{E}_6) \to A'$ be the homomorphism of algebras 
determined
on the arrows as follows:
\begin{gather*}
  \varphi(a_0) = a_0 + \theta_0 a_0 \bar{a}_2 a_2 \bar{a}_0 a_0 (\bar{a}_2 a_2)^2
  , \,
  \varphi(a_2) = a_2 + (\theta_0 + \theta_3) a_2 \bar{a}_2 a_2 \bar{a}_0 a_0 (\bar{a}_2 a_2)^2
  ,\\
  \varphi(a_k) = a_k
  , \mbox{ for } k \in \{ 1,3,4 \}
  , \ \
  \varphi(\bar{a}_l) = \bar{a}_l
  , \mbox{ for } k \in \{ 0,\dots,4 \}
  .
\end{gather*}
We prove that $\varphi$ is well defined.
Indeed, we have the equalities
\begin{align*}
    \varphi(a_0 \bar{a}_0)
    &= \varphi(a_0) \varphi(\bar{a}_0)
     = a_0 \bar{a}_0
       + \theta_0 a_0 \bar{a}_2 a_2 \bar{a}_0 a_0 (\bar{a}_2 a_2)^2 \bar{a}_0
     = 0,
    \\
    \varphi(a_1 \bar{a}_1)
    &= \varphi(a_1) \varphi(\bar{a}_1) = a_1 \bar{a}_1 = 0, 
    \\
    \varphi(\bar{a}_{4} a_{4}) &=
    \varphi(\bar{a}_{4}) \varphi(a_{4})
    = \bar{a}_{4} a_{4} = 0 , \\
    \varphi(\bar{a}_{3} a_{3} + a_{4} \bar{a}_{4}) &=
    \varphi(\bar{a}_{3}) \varphi(a_{3}) + \varphi(a_{4}) \varphi(\bar{a}_{4})
    = \bar{a}_{3} a_{3} + a_{4} \bar{a}_{4} = 0, \\
    \varphi(\bar{a}_{1} a_{1} + a_{2} \bar{a}_{2}) &=
    \varphi(\bar{a}_{1}) \varphi(a_{1}) + \varphi(a_{2}) \varphi(\bar{a}_{2})
    \\ &=
    \bar{a}_{1} a_{1} + a_{2} \bar{a}_{2} + (\theta_0 + \theta_3) a_2 \bar{a}_2 a_2 \bar{a}_0 a_0 (\bar{a}_2 a_2)^2 \bar{a}_{2}
    \\ &=
    \bar{a}_{1} a_{1} + a_{2} \bar{a}_{2} + (\theta_0 + \theta_3) a_2 \bar{a}_2 a_2 \bar{a}_0 a_0 \bar{a}_2 (\bar{a}_1 a_1)^2
    \\ &=
    \bar{a}_{1} a_{1} + a_{2} \bar{a}_{2} = 0
\end{align*}
(in the last equation we use $a_1 \bar{a}_1 = 0$).
Moreover, by (\ref{eq:4}), we obtain
\begin{align*}
    \varphi(\bar{a}_{0} a_{0} + \bar{a}_{2} a_{2} + a_{3} \bar{a}_{3}) &=
    \varphi(\bar{a}_{0}) \varphi(a_{0}) + \varphi(\bar{a}_{2}) \varphi(a_{2}) + \varphi(a_{3}) \varphi(\bar{a}_{3})
    \\ &=
    \bar{a}_{0} a_{0}
    + \theta_0 \bar{a}_0 a_0 \bar{a}_2 a_2 \bar{a}_0 a_0 (\bar{a}_2 a_2)^2
    \\ & \ \ \
    + \bar{a}_{2} a_{2}
    + (\theta_0 + \theta_3) (\bar{a}_2 a_2)^2 \bar{a}_0 a_0 (\bar{a}_2 a_2)^2
    + a_{3} \bar{a}_{3}
%    \\ &\stackrel{\!\eqref{eq:4}\!}{=}
    \\ &=
    \bar{a}_{0} a_{0}
    + \bar{a}_{2} a_{2}
    + a_{3} \bar{a}_{3}
    + \theta_3 (\bar{a}_2 a_2)^2 \bar{a}_0 a_0 (\bar{a}_2 a_2)^2
    = 0, 
\end{align*}
so $\varphi$ is well defined.
Similarly, we construct the algebra homomorphism $\psi : A' \to P(\mathbb{E}_6)$ 
determined by on the arrows by
\begin{gather*}
  \psi(a_0) = a_0 - \theta_0 a_0 \bar{a}_2 a_2 \bar{a}_0 a_0 (\bar{a}_2 a_2)^2
  , \,
  \psi(a_2) = a_2 - (\theta_0 + \theta_3) a_2 \bar{a}_2 a_2 \bar{a}_0 a_0 (\bar{a}_2 a_2)^2
  ,\\
  \psi(a_k) = a_k
  , \mbox{ for } k \in \{ 1,3,4 \}
  , \ \
  \psi(\bar{a}_l) = \bar{a}_l
  , \mbox{ for } k \in \{ 0,\dots,4 \} ,
\end{gather*}
and prove that it is well defined.
The compositions $\psi \varphi$ and $\varphi \psi$ are
the identities on $P(\mathbb{E}_6)$ and $A'$, respectively.
Hence $\psi$ and $\varphi$ are isomorphisms
and the algebras $P(\mathbb{E}_6)$ and $A'$ are isomorphic.
Therefore, the algebras $P(\mathbb{E}_6)$ and $A$ are also isomorphic.
\end{proof}

\begin{lem}
Let $n > 4$ be an odd integer
and let $A$ be a self-injective algebra socle equivalent to 
the preprojective algebra $P(\mathbb{D}_n)$. 
Then the algebras $A$ and $P(\mathbb{D}_n)$ are isomorphic.
\end{lem}

\begin{proof}
Let $n = 2m+1$ for some integer $m \geq 2$.
The algebra $A$ is isomorphic 
to a bound quiver algebra $A'$ 
of the quiver $Q_{\mathbb{D}_n}$
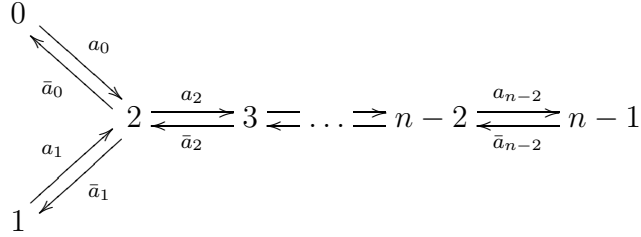
\[
    \xymatrix@C=1pc{
        0 \ar@<.5ex>^{a_0}[rrd] \\
         && 2 \ar@<.5ex>^{\bar{a}_0}[llu] \ar@<.5ex>^{\bar{a}_1}[lld] \ar@<.5ex>^{a_2}[rr] && 3
        \ar@<.5ex>^{\bar{a}_2}[ll]
         \ar@{-}@<.5ex>[r] & \ar@<.5ex>[l] \dots \ar@<.5ex>[r] & \ar@{-}@<.5ex>[l]
         n-2 \ar@<.5ex>^{a_{n-2}}[rr] &&
         n - 1 \ar@<.5ex>^(.5){\bar{a}_{n-2}}[ll] \\
        1 \ar@<.5ex>^{a_1}[rru] \\
    }
\]
bound by the relations:
\begin{gather*}
  a_0 \bar{a}_0 = 0 
  , \ \ 
  a_1 \bar{a}_1 = 0
 ,\\
  \bar{a}_0 a_0 + \bar{a}_1 a_1 +  a_2 \bar{a}_2
  +
  \theta_2  
   \bar{a}_1 a_1 a_2 \dots a_{n-2} \bar{a}_{n-2} \dots \bar{a}_3 \bar{a}_2
   = 0
 ,\\
  a_0 (\bar{a}_1 a_1 + a_2 \bar{a}_2) = 0
  , \ \ 
  a_1 (\bar{a}_0 a_0 + a_2 \bar{a}_2) = 0
  , \ \ 
  \bar{a}_2 (\bar{a}_0 a_0 + \bar{a}_1 a_1 + a_2 \bar{a}_2) = 0
 ,\\
  (\bar{a}_1 a_1 + a_2 \bar{a}_2) \bar{a}_0 = 0
  , \ \ 
  (\bar{a}_0 a_0 + a_2 \bar{a}_2) \bar{a}_1 = 0
  , \ \ 
  (\bar{a}_0 a_0 + \bar{a}_1 a_1 + a_2 \bar{a}_2) a_2 = 0
 ,\\
  \bar{a}_{k-1} a_{k-1} +  a_{k} \bar{a}_{k}
  +
  \theta_k \bar{a}_{k-1} \dots \bar{a}_2 
           \bar{a}_1 a_1 a_2 \dots a_{n-2} \bar{a}_{n-2} \dots \bar{a}_k
   = 0
 ,\\
  a_{k-1} (\bar{a}_{k-1} a_{k-1} + a_k \bar{a}_k) = 0
  , \ \ 
  (\bar{a}_{k-1} a_{k-1} + a_k \bar{a}_k) \bar{a}_{k-1} = 0
 ,\\
  \bar{a}_k (\bar{a}_{k-1} a_{k-1} + a_k \bar{a}_k) = 0
  , \ \ 
  (\bar{a}_{k-1} a_{k-1} + a_k \bar{a}_k) a_k = 0
  , \mbox{ for } k \in \{ 3, \dots, n-2 \}
 ,\\
  \bar{a}_{n-2} a_{n-2}
  +
  \theta_{n-2} \bar{a}_{n-2} \dots \bar{a}_2  
               \bar{a}_1 a_1 a_2 \dots a_{n-2} 
   = 0
 ,\\
  \bar{a}_{n-2} a_{n-2} \bar{a}_{n-2}
   = 0
  , \ \ 
  a_{n-2} \bar{a}_{n-2} a_{n-2}
   = 0
  ,
\end{gather*}
for some coefficients $\theta_2, \dots, \theta_{n-1} \in K$.
We will show that the algebras $P(\mathbb{D}_{n})$ and $A'$ are isomorphic.

Observe first, that 
for any positive integer $k$,
the following equalities hold
in the algebras
$P(\mathbb{D}_{n})$ and $A'$
\[
  a_0 (a_2 \bar{a}_2)^{2k} \bar{a}_0
  = 0 
\quad   
  \mbox{ and }
\quad   
  a_1 (a_2 \bar{a}_2)^{2k} \bar{a}_1
  = 0 
  .
\]
Indeed, from the equalities $a_0 \bar{a}_0 = 0$ and $a_1 \bar{a}_1 = 0$, 
we obtain
\begin{align*}
  a_0 (a_2 \bar{a}_2)^{2k} \bar{a}_0
  &=
    a_0 (- \bar{a}_0 a_0 - \bar{a}_1 a_1)^{2k} \bar{a}_0
   =
    a_0 \bar{a}_1 a_1 (\bar{a}_0 a_0 + \bar{a}_1 a_1)^{2k-1} \bar{a}_0
\\&=
    a_0 \bar{a}_0 a_0 \bar{a}_1 a_1 (\bar{a}_0 a_0 + \bar{a}_1 a_1)^{2k-2} \bar{a}_0
   =
    \dots
   =
    a_0 (\bar{a}_1 a_1 \bar{a}_0 a_0)^k \bar{a}_0
   = 0
\end{align*}
and
\begin{align*}
  a_1 (a_2 \bar{a}_2)^{2k} \bar{a}_1
  &=
    a_1 (\bar{a}_0 a_0 + \bar{a}_1 a_1)^{2k} \bar{a}_1
   =
    a_1 \bar{a}_0 a_0 (\bar{a}_0 a_0 + \bar{a}_1 a_1)^{2k-1} \bar{a}_1
\\&=
    a_1 \bar{a}_0 a_0 \bar{a}_1 a_1 (\bar{a}_0 a_0 + \bar{a}_1 a_1)^{2k-2} \bar{a}_1
   =
    \dots
   =
    a_1 (\bar{a}_0 a_0 \bar{a}_1 a_1)^k \bar{a}_1
   = 0
   .
\end{align*}
Hence we obtain also
\begin{align*}
  a_1 a_2 \dots a_{n-2} \bar{a}_{n-2} \dots \bar{a}_2 \bar{a}_1
  &=
  (-1)^{\frac{(n-3)(n-4)}{2}} 
  a_1 (a_2 \bar{a}_2)^{n-3} \bar{a}_1
\\&
  =
  (-1)^{2m^2-5m+3} 
  a_1 (a_2 \bar{a}_2)^{2(m-1)} \bar{a}_1
  = 0
   .
\end{align*}

Now we will construct an algebra isomorphism from $P(\mathbb{D}_{n})$ to $A'$.

%We may define isomorphism
Let
$\varphi : P(\mathbb{D}_{n}) \to A'$
be the homomorphism
given on the arrows as follows
\begin{gather*}
  \varphi(a_0) = a_0
  , \, 
  \qquad
  \varphi(a_1) = a_1 + \left( \sum_{i=2}^{n-1} (-1)^{i} \theta_i \right) 
         a_1 a_2 \dots a_{n-2} \bar{a}_{n-2} \dots \bar{a}_2
  ,\\
  \varphi(a_k) = a_k + \left( \sum_{i=k+1}^{n-1} (-1)^{i+k+1} \theta_i \right) 
         \bar{a}_{k-1} \dots \bar{a}_2 \bar{a}_1 a_1 a_2 \dots a_{n-2} \bar{a}_{n-2} \dots \bar{a}_{k+1}
  ,\\
  \varphi(\bar{a}_l) = \bar{a}_l 
  , 
\quad
  \mbox{ for }  
\quad
  k \in \{ 2,\dots,n-2 \}, l \in \{ 0,\dots,n-1 \} 
  .
\end{gather*}
We claim that $\varphi$ is well defined.
Indeed, we have the equalities
\begin{align*}
    \varphi(a_0 \bar{a}_0) 
    &= \varphi(a_0) \varphi(\bar{a}_0) 
     = a_0 \bar{a}_0 
     = 0, 
    \\
    \varphi(a_1 \bar{a}_1) 
    &= \varphi(a_1) \varphi(\bar{a}_1) 
     \\ &= a_1 \bar{a}_1 
       +\left( \sum_{i=2}^{n-1} (-1)^{i} \theta_i \right) 
         a_1 a_2 \dots a_{n-2} \bar{a}_{n-2} \dots \bar{a}_2 \bar{a}_1
     \\ &= a_1 \bar{a}_1 
     = 0,
    \\
%\end{align*}\begin{align*}
    \varphi(\bar{a}_{0} a_{0} + \bar{a}_{1} a_{1} &+ a_{2} \bar{a}_{2}) = 
    \varphi(\bar{a}_{0}) \varphi(a_{0}) + 
    \varphi(\bar{a}_{1}) \varphi(a_{1}) + \varphi(a_{2}) \varphi(\bar{a}_{2})
    \\ &= 
    \bar{a}_{0} a_{0} + \bar{a}_{1} a_{1} + a_{2} \bar{a}_{2}
    \\ & \ \ 
     + \left( \sum_{i=2}^{n-1} (-1)^{i} \theta_i + \sum_{i=3}^{n-1} (-1)^{i+3} \theta_i \right) 
         \bar{a}_1 a_1 a_2 \dots a_{n-2} \bar{a}_{n-2} \dots \bar{a}_2
    \\ &= 
    \bar{a}_{0} a_{0} + \bar{a}_{1} a_{1} + a_{2} \bar{a}_{2}
     + \theta_2
         \bar{a}_1 a_1 a_2 \dots a_{n-2} \bar{a}_{n-2} \dots \bar{a}_2
    \\ &= 0 ,
    \\
%\end{align*}\begin{align*}
    \varphi(\bar{a}_{k-1} a_{k-1} + &a_{k} \bar{a}_{k}) = 
    \varphi(\bar{a}_{k-1}) \varphi(a_{k-1}) + \varphi(a_{k}) \varphi(\bar{a}_{k})
    \\ &= 
    \bar{a}_{k-1} a_{k-1} + a_{k} \bar{a}_{k}
     \\ & \ \  
         + \left( \sum_{i=k}^{n-1} (-1)^{i+k} \theta_i + \sum_{i=k+1}^{n-1} (-1)^{i+k+1} \theta_i \right) 
     \\ & \ \qquad  
         \bar{a}_{k-1} \dots \bar{a}_2 \bar{a}_1 a_1 a_2 \dots a_{n-2} \bar{a}_{n-2} \dots \bar{a}_{k}
    \\ &= 
    \bar{a}_{k-1} a_{k-1} + a_{k} \bar{a}_{k}
         + \theta_k 
         \bar{a}_{k-1} \dots \bar{a}_2 \bar{a}_1 a_1 a_2 \dots a_{n-2} \bar{a}_{n-2} \dots \bar{a}_{k}
         ,
    \\ &= 0 ,
    \\
    \varphi(\bar{a}_{n-2} a_{n-2}) &= 
    \varphi(\bar{a}_{n-2}) \varphi(a_{n-2}))
    \\ &= 
    \bar{a}_{n-2} a_{n-2}
         + \theta_{n-1} 
         \bar{a}_{n-2} \dots \bar{a}_2 \bar{a}_1 a_1 a_2 \dots a_{n-2}
     = 0
\end{align*}
(with $k \in \{3,\dots,n-2\}$), so $\varphi$ is well defined.

Similarly, we construct the homomorphism 
of algebras $\psi : A \to P(\mathbb{D}_n)$ 
given on the arrows by
\begin{gather*}
  \psi(a_0) = a_0
  , \, 
  \quad
  \psi(a_1) = a_1 - \left( \sum_{i=2}^{n-1} (-1)^{i} \theta_i \right) 
         a_1 a_2 \dots a_{n-2} \bar{a}_{n-2} \dots \bar{a}_2
  ,\\
  \psi(a_k) = a_k - \left( \sum_{i=k+1}^{n-1} (-1)^{i+k+1} \theta_i \right) 
         \bar{a}_{k-1} \dots \bar{a}_2 \bar{a}_1 a_1 a_2 \dots a_{n-2} \bar{a}_{n-2} \dots \bar{a}_{k+1}
  ,\\
  \psi(\bar{a}_l) = \bar{a}_l 
  \quad
  , \mbox{ for }  k \in \{ 2,\dots,n-2 \}, l \in \{ 0,\dots,n-1 \} 
  ,
\end{gather*}
and prove that it is well defined.

Then we observe that $\psi \varphi$ and $\varphi \psi$ 
are the identities on the algebras $P(\mathbb{D}_n)$ and $A'$,
respectively.
Hence $\psi$ and $\varphi$ are isomorphisms
and the algebras $P(\mathbb{D}_n)$ and $A'$ are isomorphic.
Therefore, the algebras $P(\mathbb{D}_n)$ and $A$ are also isomorphic.
\end{proof}

This concludes the proof of Proposition~\ref{prop:no-def}.

\section{Algebras with proper socle deformations}
\label{sec:only-canoninical}

The aim of this section is to prove the following proposition.

\begin{prop}
\label{prop:iso}
Let $A$ be a self-injective algebra socle equivalent to
the preprojective algebra $P(\Delta)$ of generalized Dynkin type 
$\Delta \in \{ \mathbb{D}_{2m}, 
\mathbb{E}_7, \mathbb{E}_8, \mathbb{L}_{n} \}$, 
for some $m,n \geq 2$.
Then $A$ is isomorphic to 
one of the algebras $P(\Delta)$ or $P^*(\Delta)$.
\end{prop}

We note that at this moment we do not discuss
whether and when the algebras 
$P(\Delta)$ and $P^*(\Delta)$
are isomorphic. 
It will be discussed in the next two sections.

We divide the proof of the 
proposition into four cases.

\begin{lem}
Let $n \geq 4$ be an even integer and let $A$ be
a self-injective algebra socle equivalent to the preprojective algebra
of type $\mathbb{D}_n$.
Then $A$ is isomorphic to one of the algebras 
$P(\mathbb{D}_n)$ or $P^{*}(\mathbb{D}_n)$.
\end{lem}

\begin{proof}
Let $n = 2 m$ for some integer $m \geq 2$. 
The algebra $A$ is isomorphic 
to a bound quiver algebra $A'$ 
of the quiver $Q_{\mathbb{D}_n}$
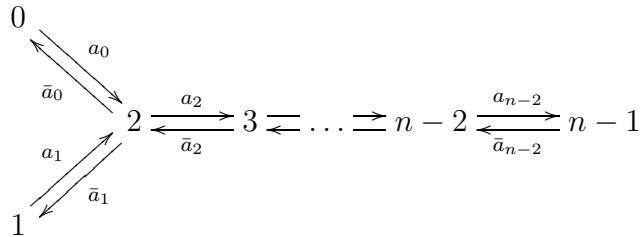
\[
    \xymatrix@C=1pc{
        0 \ar@<.5ex>^{a_0}[rrd] \\
         && 2 \ar@<.5ex>^{\bar{a}_0}[llu] \ar@<.5ex>^{\bar{a}_1}[lld] \ar@<.5ex>^{a_2}[rr] && 3
        \ar@<.5ex>^{\bar{a}_2}[ll]
         \ar@{-}@<.5ex>[r] & \ar@<.5ex>[l] \dots \ar@<.5ex>[r] & \ar@{-}@<.5ex>[l]
         n-2 \ar@<.5ex>^{a_{n-2}}[rr] &&
         n - 1 \ar@<.5ex>^(.5){\bar{a}_{n-2}}[ll] \\
        1 \ar@<.5ex>^{a_1}[rru] \\
    }
\]
bound by the relations:
\begin{gather*}
  a_0 \bar{a}_0 + \theta_0 a_0 a_2 a_3 \dots a_{n-2} \bar{a}_{n-2} \dots \bar{a}_2 \bar{a}_0
  = 0 
  , \ \ 
  a_0 \bar{a}_0 a_0 = 0
  , \ \ 
  \bar{a}_0 a_0 \bar{a}_0 = 0
  ,\\ 
  a_1 \bar{a}_1 + \theta_1 a_1  a_2 a_3 \dots a_{n-2} \bar{a}_{n-2} \dots \bar{a}_2 \bar{a}_1 = 0
  , \ \ 
  a_1 \bar{a}_1 a_1 = 0
  , \ \ 
  \bar{a}_1 a_1 \bar{a}_1 = 0
 ,\\
  \bar{a}_0 a_0 + \bar{a}_1 a_1 +  a_2 \bar{a}_2
  +
  \theta_2  
      \bar{a}_1 a_1 a_2 \dots a_{n-2} \bar{a}_{n-2} \dots \bar{a}_3 \bar{a}_2
   = 0
 ,\\
  a_0 (\bar{a}_1 a_1 + a_2 \bar{a}_2) = 0
  , \ \ 
  a_1 (\bar{a}_0 a_0 + a_2 \bar{a}_2) = 0
  , \ \ 
  \bar{a}_2 (\bar{a}_0 a_0 + \bar{a}_1 a_1 + a_2 \bar{a}_2) = 0
 ,\\
  (\bar{a}_1 a_1 + a_2 \bar{a}_2) \bar{a}_0 = 0
  , \ \ 
  (\bar{a}_0 a_0 + a_2 \bar{a}_2) \bar{a}_1 = 0
  , \ \ 
  (\bar{a}_0 a_0 + \bar{a}_1 a_1 + a_2 \bar{a}_2) a_2 = 0
 ,
 \\
%\end{gather*}\begin{gather*}
  \bar{a}_{k-1} a_{k-1} +  a_{k} \bar{a}_{k}
  +
  \theta_k \bar{a}_{k-1} \dots \bar{a}_2 
           \bar{a}_1 a_1 a_2 \dots a_{n-2} \bar{a}_{n-2} \dots \bar{a}_k
   = 0
 ,\\
  a_{k-1} (\bar{a}_{k-1} a_{k-1} + a_k \bar{a}_k) = 0
  , \ \ 
  (\bar{a}_{k-1} a_{k-1} + a_k \bar{a}_k) \bar{a}_{k-1} = 0
 ,\\
  \bar{a}_k (\bar{a}_{k-1} a_{k-1} + a_k \bar{a}_k) = 0
  , \ \ 
  (\bar{a}_{k-1} a_{k-1} + a_k \bar{a}_k) a_k = 0
  , \mbox{ for } k \in \{ 3, \dots, n-2 \}
 ,\\
  \bar{a}_{n-2} a_{n-2}
  +
  \theta_{n-2} \bar{a}_{n-2} \dots \bar{a}_2  
               \bar{a}_1 a_1 a_2 \dots a_{n-2} 
   = 0
 ,\\
  \bar{a}_{n-2} a_{n-2} \bar{a}_{n-2} = 0
  , \ \ 
  a_{n-2} \bar{a}_{n-2} a_{n-2} = 0
  ,
\end{gather*}
for some coefficients
$\theta_0, \dots, \theta_{n-1} \in K$.
Moreover, let $A''$ be the bound quiver algebra 
of the quiver $Q_{\mathbb{D}_n}$
bound by the relations:
\begin{gather*}
  a_0 \bar{a}_0 = 0 
  , \ \ 
  a_1 \bar{a}_1 = 0
 ,\\
  \bar{a}_0 a_0 + \bar{a}_1 a_1 +  a_2 \bar{a}_2
  +
  \theta 
  \bar{a}_1 a_1 a_2 \dots a_{n-2} \bar{a}_{n-2} \dots \bar{a}_3 \bar{a}_2
   = 0
 ,\\
  a_0 (\bar{a}_1 a_1 + a_2 \bar{a}_2) = 0
  , \ \ 
  a_1 (\bar{a}_0 a_0 + a_2 \bar{a}_2) = 0
  , \ \ 
  \bar{a}_2 (\bar{a}_0 a_0 + \bar{a}_1 a_1 + a_2 \bar{a}_2) = 0
 ,
 \\
%\end{gather*}\begin{gather*}
  (\bar{a}_1 a_1 + a_2 \bar{a}_2) \bar{a}_0 = 0
  , \ \ 
  (\bar{a}_0 a_0 + a_2 \bar{a}_2) \bar{a}_1 = 0
  , \ \ 
  (\bar{a}_0 a_0 + \bar{a}_1 a_1 + a_2 \bar{a}_2) a_2 = 0
 ,\\
  \bar{a}_{k-1} a_{k-1} +  a_{k} \bar{a}_{k}
   = 0
  , \mbox{ for } k \in \{ 3, \dots, n-2 \}
  , \ \ 
% ,\\
  \bar{a}_{n-2} a_{n-2}
   = 0
  ,
\end{gather*}
where $\theta = \sum_{i=0}^{n-2} (-1)^i \theta_i$. 

We will show first that the algebras  $A'$ and $A''$ are isomorphic.

Observe 
that the following equalities hold 
in the algebras $A'$ and $A''$
\begin{align*}
  a_2 \bar{a}_2
  a_2 a_3 \dots a_{n-2} \bar{a}_{n-2} \dots \bar{a}_2 
  & = 
  - a_2 a_3 \bar{a}_3  a_3 \dots a_{n-2} \bar{a}_{n-2} \dots \bar{a}_2 
 \\ & = 
  a_2 a_3 a_4 \bar{a}_4
  a_4 \dots a_{n-2} \bar{a}_{n-2} \dots \bar{a}_2 
  \\&
  = \dots  
  = (-1)^{n-3} a_2 \dots a_{n-1} \bar{a}_{n-1} a_{n-1} a_{n-2} \bar{a}_{n-2} \dots \bar{a}_2 
  \\&= (-1)^{n-4} a_2 \dots a_{n-1} a_{n-2} \bar{a}_{n-2} a_{n-2} \bar{a}_{n-2} \dots \bar{a}_2 
  = 
  0 
   .
\end{align*}
Then, applying the relation 
$(\bar{a}_0 a_0 + \bar{a}_1 a_1 + a_2 \bar{a}_2) a_2 = 0$,
we obtain the equality
\[
 (\bar{a}_0 a_0 + \bar{a}_1 a_1) a_2 \dots a_{n-2} \bar{a}_{n-2} \dots \bar{a}_2
 = 0 
 .
\]
Similarly, we prove the dual equality
\[
  a_2 \dots a_{n-2} \bar{a}_{n-2} \dots \bar{a}_2 (\bar{a}_0 a_0 + \bar{a}_1 a_1)
 = 0
  .
\]

Let $\varphi : A'' \to A'$ be the homomorphism of algebras 
defined on the arrows as follows
\begin{gather*}
  \varphi(a_0) = a_0 
   + \theta_0 
         a_0 a_2 \dots a_{n-2} \bar{a}_{n-2} \dots \bar{a}_2
  ,\\
  \varphi(a_1) = a_1 
         + \theta_1 
         a_1 a_2 \dots a_{n-2} \bar{a}_{n-2} \dots \bar{a}_2
  ,\\
  \varphi(a_k) = a_k + \left( \sum_{i=k+1}^{n-1} (-1)^{i+k+1} \theta_i \right) 
         \bar{a}_{k-1} \dots \bar{a}_2 \bar{a}_1 a_1 a_2 \dots a_{n-2} \bar{a}_{n-2} \dots \bar{a}_{k+1}
  ,\\
  \varphi(\bar{a}_l) = \bar{a}_l 
  , \mbox{ for }  k \in \{ 2,\dots,n-2 \}, l \in \{ 0,\dots,n-1 \} 
  .
\end{gather*}
We show that $\varphi$ is well defined.
Indeed, we have the equalities
\begin{align*}
    \varphi(a_0 \bar{a}_0) 
    &= \varphi(a_0) \varphi(\bar{a}_0) 
     = a_0 \bar{a}_0 
       + \theta_0 
         a_0 a_2 \dots a_{n-2} \bar{a}_{n-2} \dots \bar{a}_2 \bar{a}_0 
     = 0, 
    \\
    \varphi(a_1 \bar{a}_1) 
    &= \varphi(a_1) \varphi(\bar{a}_1) 
%     \\ &
      = a_1 \bar{a}_1 
       + \theta_1 
         a_1 a_2 \dots a_{n-2} \bar{a}_{n-2} \dots \bar{a}_2 \bar{a}_1
     = 0,
    \\
    \varphi(\bar{a}_{0} a_{0} + \bar{a}_{1} a_{1} + a_{2} \bar{a}_{2}
      &+
  \theta 
  \bar{a}_1 a_1 a_2 \dots a_{n-2} \bar{a}_{n-2} \dots \bar{a}_3 \bar{a}_2
    ) 
    \\&=
    \varphi(\bar{a}_{0}) \varphi(a_{0}) + 
    \varphi(\bar{a}_{1}) \varphi(a_{1}) + \varphi(a_{2}) \varphi(\bar{a}_{2})
     \\ & \ \ 
     + \theta  
      \varphi(\bar{a}_1) \varphi(a_1) \varphi(a_2) \dots \varphi(a_{n-2}) \varphi(\bar{a}_{n-2}) \dots \varphi(\bar{a}_3) \varphi(\bar{a}_2)
    \\ &= 
    \bar{a}_{0} a_{0} + \bar{a}_{1} a_{1} + a_{2} \bar{a}_{2}
    \\ & \ \ \ 
     + \theta_0
         \bar{a}_0 a_0 a_2 \dots a_{n-2} \bar{a}_{n-2} \dots \bar{a}_2
    \\ & \ \ \ 
     + \left( \theta + \theta_1 + \sum_{i=3}^{n-1} (-1)^{i+3} \theta_i \right) 
         \bar{a}_1 a_1 a_2 \dots a_{n-2} \bar{a}_{n-2} \dots \bar{a}_2
    \\ &= 
    \bar{a}_{0} a_{0} + \bar{a}_{1} a_{1} + a_{2} \bar{a}_{2}
     + \theta_2
         \bar{a}_1 a_1 a_2 \dots a_{n-2} \bar{a}_{n-2} \dots \bar{a}_2
    \\ &= 0 ,
    \\
%\end{align*}\begin{align*}
    \varphi(\bar{a}_{k-1} a_{k-1} + a_{k} \bar{a}_{k}) &= 
    \varphi(\bar{a}_{k-1}) \varphi(a_{k-1}) + \varphi(a_{k}) \varphi(\bar{a}_{k})
    \\ &= 
    \bar{a}_{k-1} a_{k-1} + a_{k} \bar{a}_{k}
     \\ & \ \  
         + \left( \sum_{i=k}^{n-1} (-1)^{i+k} \theta_i + \sum_{i=k+1}^{n-1} (-1)^{i+k+1} \theta_i \right) 
     \\ & \ \qquad  
         \bar{a}_{k-1} \dots \bar{a}_2 \bar{a}_1 a_1 a_2 \dots a_{n-2} \bar{a}_{n-2} \dots \bar{a}_{k}
    \\ &= 
    \bar{a}_{k-1} a_{k-1} + a_{k} \bar{a}_{k}
         + \theta_k 
         \bar{a}_{k-1} \dots \bar{a}_2 \bar{a}_1 a_1 a_2 \dots a_{n-2} \bar{a}_{n-2} \dots \bar{a}_{k}
    \\ &= 0 
         ,
    \\
    \varphi(\bar{a}_{n-2} a_{n-2}) &= 
    \varphi(\bar{a}_{n-2}) \varphi(a_{n-2}))
    \\ &= 
    \bar{a}_{n-2} a_{n-2}
         + \theta_{n-1} 
         \bar{a}_{n-2} \dots \bar{a}_2 \bar{a}_1 a_1 a_2 \dots a_{n-2}
       = 0 
         ,
\end{align*}
so $\varphi$ is well defined.
Similarly, we define the homomorphism of algebras
$\psi : A' \to A''$ 
given on the arrows by
\begin{gather*}
  \psi(a_0) = a_0 
   - \theta_0 
         a_0 a_2 \dots a_{n-2} \bar{a}_{n-2} \dots \bar{a}_2
  ,\\
  \psi(a_1) = a_1 
         - \theta_1 
         a_1 a_2 \dots a_{n-2} \bar{a}_{n-2} \dots \bar{a}_2
  ,
%  \\
\end{gather*}\begin{gather*}
  \psi(a_k) = a_k - \left( \sum_{i=k+1}^{n-1} (-1)^{i+k+1} \theta_i \right) 
         \bar{a}_{k-1} \dots \bar{a}_2 \bar{a}_1 a_1 a_2 \dots a_{n-2} \bar{a}_{n-2} \dots \bar{a}_{k+1}
  ,
  \\
  \psi(\bar{a}_l) = \bar{a}_l 
  , \mbox{ and }  k \in \{ 2,\dots,n-2 \}, l \in \{ 0,\dots,n-1 \} ,
\end{gather*}
and prove that it is well defined.
The compositions $\psi \varphi$ and $\varphi \psi$ 
are the identities on $A''$ and $A'$, respectively,
so $\psi$ and $\varphi$ are isomorphisms, 
and hence the algebras $A'$ and $A''$ are isomorphic.
Therefore, the algebras $A$ and $A''$ are also isomorphic.

Therefore, it suffices to prove that 
$A''$ is isomorphic to 
$P(\mathbb{D}_n)$ or
$P^{*}(\mathbb{D}_n)$.

If $\theta = 0$, then $A''$ is isomorphic, by definition, 
to the preprojective algebra $P(\mathbb{D}_n)$.
So assume that $\theta \neq 0$. 
Then there exists
an element $\lambda \in K \setminus \{0 \}$ such that
$\lambda^{2n-3} = \theta$.
It can be easily seen that 
the homomorphism
$\phi : P^{*}(\mathbb{D}_n) \to A''$ 
given on the arrows by 
\begin{gather*}
  \phi(a_k) = \lambda a_k
  , \, 
  \psi(\bar{a}_k) = \lambda \bar{a}_k 
  , \mbox{ for }  k \in \{ 0,\dots,n-1 \} ,
\end{gather*}
is an algebra isomorphism, and hence 
the algebras 
$A''$ and $P^{*}(\mathbb{D}_n)$ are isomorphic.
\end{proof}

\begin{lemma}
\label{lem:eqLn}
Let $n \geq 2$ be a natural number and let
$A$ be a self-injective algebra socle equivalent 
to the preprojective algebra $P(\mathbb{L}_n)$.
Then $A$ is isomorphic to one of the algebras 
$P(\mathbb{L}_n)$ or $P^{*}(\mathbb{L}_n)$.
\end{lemma}

\begin{proof}
The algebra $A$ is isomorphic to a bound quiver algebra $A'$ 
of the quiver $Q_{\mathbb{L}_n}$
\[
    \xymatrix@C=1pc{
        0
        \ar `ld_u[] `_rd[]^{\varepsilon = \bar{\varepsilon}} []
        \ar@<.5ex>^{a_0}[rr] && 1 \ar@<.5ex>^(.45){\bar{a}_0}[ll] \ar@<.5ex>^{a_1}[rr] && 
        2 \ar@<.5ex>^(.45){\bar{a}_1}[ll]
        \ar@{-}@<.5ex>[r] & \ar@<.5ex>[l] \dots \ar@<.5ex>[r] & \ar@{-}@<.5ex>[l]
        n-2 \ar@<.5ex>^{a_{n-2}}[rr] && n - 1 \ar@<.5ex>^(.45){\bar{a}_{n-2}}[ll] \\
    } 
\]
bound by the relations:
\begin{gather*}
    \varepsilon^2 + a_0 \bar{a}_0 + \theta_0 \varepsilon^{2n-1} = 0
 ,\ \ 
    \varepsilon^{2n} = 0
 ,\\
  \bar{a}_{k-1} a_{k-1} +  a_{k} \bar{a}_{k}
  +
  \theta_k \bar{a}_{k-1} \dots \bar{a}_1 \bar{a}_0 \varepsilon a_0 
            a_1 \dots a_{n-2} \bar{a}_{n-2} \dots \bar{a}_k
   = 0
 ,\\
  a_{k-1} (\bar{a}_{k-1} a_{k-1} + a_k \bar{a}_k) = 0
  , \ \ 
  (\bar{a}_{k-1} a_{k-1} + a_k \bar{a}_k) \bar{a}_{k-1} = 0
 ,\\
  \bar{a}_k (\bar{a}_{k-1} a_{k-1} + a_k \bar{a}_k) = 0
  , \ \ 
  (\bar{a}_{k-1} a_{k-1} + a_k \bar{a}_k) a_k = 0
  , \mbox{ for } k \in \{ 1, \dots, n-2 \}
 ,\\
  \bar{a}_{n-2} a_{n-2}
  +
  \theta_{n-2} \bar{a}_{n-2} \dots \bar{a}_1 \bar{a}_0 \varepsilon a_0 
            a_1 \dots a_{n-2} 
   = 0
 ,\\
  \bar{a}_{n-2} a_{n-2} \bar{a}_{n-2} = 0
  , \ \ 
  a_{n-2} \bar{a}_{n-2} a_{n-2} = 0
  ,
\end{gather*}
for some coefficients $\theta_0, \dots, \theta_{n-1} \in K$.
Moreover, let $A''$ be the bound quiver algebra 
of the quiver $Q_{\mathbb{L}_n}$
bound by the relations:
\begin{gather*}
    \varepsilon^2 + a_0 \bar{a}_0 + \theta \varepsilon^{2n-1} = 0,
 \ \ 
    \varepsilon^{2n} = 0,
 \ \ 
    \bar{a}_{n-2} a_{n-2} = 0, 
 \\
    \bar{a}_i a_i + a_{i+1} \bar{a}_{i+1} = 0
    \mbox{ for } i \in \{0,\dots, n-3\} .
\end{gather*}
with $\theta = \theta_0 - \sum_{i=1}^{n-1} (-1)^{i+\frac{(n-2)(n-1)}{2}} \theta_i$.

Observe first that, in the algebras $A'$ and $A''$,
we have the equalities
\begin{align*}
  \varepsilon^{2n-1} 
     &= (-1)^{n-1} \varepsilon (a_0 \bar{a}_0)^{n-1}
      = (-1)^{(n-1)+(n-2)} \varepsilon a_0 (a_1 \bar{a}_1)^{n-2} \bar{a}_0 
   \\&= \dots 
      = (-1)^{\frac{(n-2)(n-1)}{2}} \varepsilon a_0 a_1 \dots a_{n-2} \bar{a}_{n-2} \dots \bar{a}_1 \bar{a}_0 
   ,
\end{align*}
and hence the equality
\begin{align*}
       \varepsilon a_0 a_1 \dots a_{n-2} \bar{a}_{n-2} \dots \bar{a}_1 \bar{a}_0 
    &= (-1)^{\frac{(n-2)(n-1)}{2}} 
       \varepsilon^{2n-1}
   .
\end{align*}
Similarly, we obtain the equality
\begin{align*}
       a_0 a_1 \dots a_{n-2} \bar{a}_{n-2} \dots \bar{a}_1 \bar{a}_0 \varepsilon 
    &= (-1)^{\frac{(n-2)(n-1)}{2}} 
       \varepsilon^{2n-1}
   .
\end{align*}

We will show that the algebras  $A'$ and $A''$ are isomorphic.
Let $\varphi : A'' \to A'$ be the homomorphism of algebras 
determined on the arrows by
\begin{gather*}
  \varphi(a_k) = a_k + \left( \sum_{i=k+1}^{n-1} (-1)^{i+k+1} \theta_i \right) 
         \bar{a}_{k-1} \dots \bar{a}_1 \bar{a}_0 \varepsilon a_0 a_1 \dots a_{n-2} \bar{a}_{n-2} \dots \bar{a}_{k+1}
  ,\\
  \varphi(\bar{a}_k) = \bar{a}_k 
  , \mbox{ for }  k \in \{ 0,\dots,n-2 \}
  , \ \ 
\quad
  \varphi(\varepsilon) = \varepsilon 
  .
\end{gather*}
We claim that $\varphi$ is well defined.
Indeed, we have the equalities
\begin{align*}
    \varphi(\varepsilon^2 + a_0 \bar{a}_0 
   + \theta \varepsilon^{2n-1}) 
    &= \varphi(\varepsilon)^2 + \varphi(a_0) \varphi(\bar{a}_0) + \theta \varphi(\varepsilon)^{2n-1}
  \\&= \varepsilon^2 + a_0 \bar{a}_0 
       + \theta \varepsilon^{2n-1}
    \\ & \ \ \ 
         + \left( \sum_{i=1}^{n-1} (-1)^{i+1} \theta_i \right) 
           \varepsilon a_0 \dots a_{n-2} \bar{a}_{n-2} \dots \bar{a}_1 \bar{a}_0 
  \\&= \varepsilon^2 + a_0 \bar{a}_0 
         +  \left(\theta + (-1)^{\frac{(n-2)(n-1)}{2}} \sum_{i=1}^{n-1} (-1)^{i} \theta_i \right) 
         \varepsilon^{2n-1}
  \\&= \varepsilon^2 + a_0 \bar{a}_0 
         + \theta_0
         \varepsilon^{2n-1}
     = 0 ,
\\
%\end{align*}\begin{align*}
    \varphi(\bar{a}_{k-1} a_{k-1} + a_{k} \bar{a}_{k}) &= 
    \varphi(\bar{a}_{k-1}) \varphi(a_{k-1}) + \varphi(a_{k}) \varphi(\bar{a}_{k})
    \\ &= 
    \bar{a}_{k-1} a_{k-1} + a_{k} \bar{a}_{k}
     \\ & \ \  
         + \left( \sum_{i=k}^{n-1} (-1)^{i+k} \theta_i + \sum_{i=k+1}^{n-1} (-1)^{i+k+1} \theta_i \right) 
     \\ & \ \qquad  
         \bar{a}_{k-1} \dots \bar{a}_0 \varepsilon a_0 \dots a_{n-2} \bar{a}_{n-2} \dots \bar{a}_{k}
    \\ &= 
    \bar{a}_{k-1} a_{k-1} + a_{k} \bar{a}_{k}
         + \theta_k 
         \bar{a}_{k-1} \dots \bar{a}_0 \varepsilon a_0 \dots a_{n-2} \bar{a}_{n-2} \dots \bar{a}_{k}
    \\ &= 0 
         ,
    \\
    \varphi(\bar{a}_{n-2} a_{n-2}) &= 
    \varphi(\bar{a}_{n-2}) \varphi(a_{n-2})
    \\ &= 
    \bar{a}_{n-2} a_{n-2}
         + \theta_{n-1} 
         \bar{a}_{n-2} \dots \bar{a}_0 \varepsilon a_0 \dots a_{n-2}
        = 0 
         ,
\end{align*}
so $\varphi$ is well defined.
Similarly, we define the homomorphism of algebras
$\psi : A' \to A''$ 
given on the arrows by
\begin{gather*}
  \psi(a_k) = a_k - \left( \sum_{i=k+1}^{n-1} (-1)^{i+k+1} \theta_i \right) 
         \bar{a}_{k-1} \dots \bar{a}_1 \bar{a}_0 \varepsilon a_0 a_1 \dots a_{n-2} \bar{a}_{n-2} \dots \bar{a}_{k+1}
  ,\\
  \psi(\bar{a}_k) = \bar{a}_k 
  , \mbox{ for }  k \in \{ 0,\dots,n-2 \} 
  , \ \ 
  \psi(\varepsilon) = \varepsilon 
  ,
\end{gather*}
and show that it is well defined.
Observe that the compositions $\psi \varphi$ and $\varphi \psi$ 
are the identities on $A''$ and $A'$, respectively.
Hence $\psi$ and $\varphi$ are isomorphisms, 
and so the algebras $A''$ and $A'$ are isomorphic.
But then the algebras 
$A''$ and $A$ are also isomorphic.

Therefore, it suffices to prove that the algebra
$A''$ is isomorphic to the algebra 
$P(\mathbb{L}_n)$, 
or to the algebra $P^{*}(\mathbb{L}_n)$.
If $\theta = 0$, then $A''$ 
is 
the preprojective algebra $P(\mathbb{L}_n)$, 
by definition.
So assume that $\theta \neq 0$. 
Then there exists $\lambda \in K$ such that
$\lambda^{2n-3} = \theta$.
Then it can be easily seen that 
the homomorphism
$\phi : P^{*}(\mathbb{L}_n) \to A''$ 
given on the arrows by 
\begin{gather*}
  \phi(\varepsilon) = \lambda \varepsilon
  , \, 
  \phi(a_k) = \lambda a_k
  , \, 
  \psi(\bar{a}_k) = \lambda \bar{a}_k 
  , \mbox{ for }  k \in \{ 0,\dots,n-1 \} ,
\end{gather*}
is an algebra isomorphism. 
Converselly, the algebras 
$A''$ and $P^{*}(\mathbb{L}_n)$ are isomorphic.
This ends the proof.
\end{proof}

\begin{lemma}
\label{lem:E7-1}
Let $A$ be a self-injective algebra socle equivalent 
to the preprojective algebra $P(\mathbb{E}_7)$.
Then $A$ is isomorphic to one of the algebras 
$P(\mathbb{E}_7)$  or  $P^{*}(\mathbb{E}_7)$.
\end{lemma}

\begin{proof}
    The algebra $A$ is isomorphic to a bound quiver 
    algebra $A'$ of the quiver $Q_{\mathbb{E}_7}$
    \[
    \xymatrix{
        &&0 \ar@<.5ex>^{a_0}[d] \\
        1 \ar@<.5ex>^{a_1}[r] & 2 \ar@<.5ex>^{\bar{a}_1}[l] \ar@<.5ex>^{a_2}[r] &
        3 \ar@<.5ex>^{\bar{a}_2}[l] \ar@<.5ex>^{a_3}[r] \ar@<.5ex>^{\bar{a}_0}[u] &
        4 \ar@<.5ex>^{\bar{a}_3}[l] \ar@<.5ex>^{a_4}[r] &
        5 \ar@<.5ex>^(.5){\bar{a}_4}[l] \ar@<.5ex>^{a_5}[r] &
        6 \ar@<.5ex>^(.5){\bar{a}_5}[l] \\
    }
    \]
    bound by the relations:
    \begin{gather*}
        a_0 \bar{a}_0 + \theta_0 
        a_0 (\bar{a}_2 a_2)^2 \bar{a}_0 a_0 \bar{a}_2 a_2 \bar{a}_0 a_0 (\bar{a}_2 a_2)^2 \bar{a}_0
        = 0 
        , \ \ 
        a_0 \bar{a}_0 a_0 = 0
        , \ \ 
        \bar{a}_0 a_0 \bar{a}_0 = 0
        ,\\ 
        a_1 \bar{a}_1 + \theta_1 
        a_1 a_2 \bar{a}_0 a_0 \bar{a}_2 a_2 \bar{a}_0 a_0 (\bar{a}_2 a_2)^2 \bar{a}_0 a_0 \bar{a}_2 \bar{a}_1
        = 0 
        , \ \ 
        a_1 \bar{a}_1 a_1 = 0
        , \ \ 
        \bar{a}_1 a_1 \bar{a}_1 = 0
         ,
        \\
        \bar{a}_0 a_0 + \bar{a}_2 a_2 + a_3 \bar{a}_3 + 
        \theta_3 
        (\bar{a}_2 a_2)^2 \bar{a}_0 a_0 \bar{a}_2 a_2 \bar{a}_0 a_0 (\bar{a}_2 a_2)^2 \bar{a}_0 a_0
        = 0
        ,
        \\
%    \end{gather*}\begin{gather*}
        a_0 (\bar{a}_2 a_2 + a_3 \bar{a}_3) = 0
        , \ \ 
        (\bar{a}_2 a_2 + a_3 \bar{a}_3) \bar{a}_0 = 0
        ,\\
        a_2 (\bar{a}_0 a_0 + \bar{a}_2 a_2 + a_3 \bar{a}_3) = 0
        , \ \ 
        (\bar{a}_0 a_0 + \bar{a}_2 a_2 + a_3 \bar{a}_3) \bar{a}_2 = 0
        ,
        \\
%    \end{gather*}\begin{gather*}
        \bar{a}_3 (\bar{a}_0 a_0 + \bar{a}_2 a_2 + a_3 \bar{a}_3) = 0
        , \ \ 
        (\bar{a}_0 a_0 + \bar{a}_2 a_2 + a_3 \bar{a}_3) a_3 = 0
        ,\\
        \bar{a}_{1} a_{1} +  a_{2} \bar{a}_{2}
        +
        \theta_2 
        \bar{a}_1 a_1 a_2 \bar{a}_0 a_0 \bar{a}_2 a_2 \bar{a}_0 a_0 (\bar{a}_2 a_2)^2 \bar{a}_0 a_0 \bar{a}_2
        = 0
        ,
%        \\
    \end{gather*}\begin{gather*}
        \bar{a}_{3} a_{3} +  a_{4} \bar{a}_{4}
        +
        \theta_4 
        \bar{a}_{3} (\bar{a}_2 a_2)^2 \bar{a}_0 a_0 \bar{a}_2 a_2 \bar{a}_0 a_0 a_{3} a_{4} a_{5} \bar{a}_{5} \bar{a}_{4}
        = 0
        ,\\
        \bar{a}_{4} a_{4} +  a_{5} \bar{a}_{5}
        +
        \theta_5 
        \bar{a}_{4} \bar{a}_{3} (\bar{a}_2 a_2)^2 \bar{a}_0 a_0 \bar{a}_2 a_2 \bar{a}_0 a_0 a_{3} a_{4} a_{5} \bar{a}_{5}
        = 0
        ,\\
        a_{k-1} (\bar{a}_{k-1} a_{k-1} + a_k \bar{a}_k) = 0
        , \ \ 
        (\bar{a}_{k-1} a_{k-1} + a_k \bar{a}_k) \bar{a}_{k-1} = 0
         ,\\
        \bar{a}_k (\bar{a}_{k-1} a_{k-1} + a_k \bar{a}_k) = 0
        , \ \ 
        (\bar{a}_{k-1} a_{k-1} + a_k \bar{a}_k) a_k = 0
        , \mbox{ for } k \in \{ 2,4,5 \}
        ,\\
        \bar{a}_{5} a_{5}
        +
        \theta_{6} 
        \bar{a}_{5} \bar{a}_{4} \bar{a}_{3} (\bar{a}_2 a_2)^2 \bar{a}_0 a_0 \bar{a}_2 a_2 \bar{a}_0 a_0 a_{3} a_{4} a_{5}
        = 0
        , \ \ 
        \bar{a}_{5} a_{5} \bar{a}_{5} = 0
        , \ \ 
        a_{5} \bar{a}_{5} a_{5} = 0
        ,
    \end{gather*}
    for some coefficients 
    $\theta_0, \dots, \theta_{6} \in K$.
    Moreover, let $A''$ be the bound quiver algebra 
    of the quiver $Q_{\mathbb{E}_7}$
    bound by the relations:
    \begin{gather*}
        a_0 \bar{a}_0 = 0
        , \ \ 
        a_1 \bar{a}_1 = 0
        , \ \ 
        \bar{a}_{k-1} a_{k-1} + a_k \bar{a}_k = 0
        , \mbox{ for } k \in \{ 2,4,5 \}
        , \ \ 
        \bar{a}_5 a_5 = 0
        ,\\
        \bar{a}_0 a_0 + \bar{a}_2 a_2 + a_3 \bar{a}_3 + 
        \theta 
        \bar{a}_2 a_2 \bar{a}_0 a_0 (\bar{a}_2 a_2)^2 \bar{a}_0 a_0 (\bar{a}_2 a_2)^2 \bar{a}_0 a_0
        = 0
        ,\\
        a_0 (\bar{a}_2 a_2 + a_3 \bar{a}_3) = 0
        , \ \ 
        (\bar{a}_2 a_2 + a_3 \bar{a}_3) \bar{a}_0 = 0
        ,\\
        a_2 (\bar{a}_0 a_0 + \bar{a}_2 a_2 + a_3 \bar{a}_3) = 0
        , \ \ 
        (\bar{a}_0 a_0 + \bar{a}_2 a_2 + a_3 \bar{a}_3) \bar{a}_2 = 0
        ,\\
        \bar{a}_3 (\bar{a}_0 a_0 + \bar{a}_2 a_2 + a_3 \bar{a}_3) = 0
        , \ \ 
        (\bar{a}_0 a_0 + \bar{a}_2 a_2 + a_3 \bar{a}_3) a_3 = 0
        ,
    \end{gather*}
    with 
    $\theta = \theta_0 + \theta_1 - \theta_2 - \theta_3 + \theta_4 - \theta_5 + \theta_6$. 
    
    To simplify the notation we abbreviate
    $x = \bar{a}_0 a_0$, $y = \bar{a}_2 a_2$ and $z = a_3 \bar{a}_3$.
    Then, in the both algebras $A'$ and $A''$, we have
    the equalities
    \begin{align*}
        x^2 &= (\bar{a}_0 a_0)^2 = \bar{a}_0 (a_0 \bar{a}_0) a_0 = 0 , \\
        y^3 &= (\bar{a}_2 a_2)^3 = \bar{a}_2 \bar{a}_1 (a_1 \bar{a}_1) a_1 a_2 = 0 , \\
        z^4 &= (a_3 \bar{a}_3)^4 
        = - a_3 a_4 a_5 (\bar{a}_5 a_5) \bar{a}_5 \bar{a}_4 \bar{a}_3 = 0 
        .
    \end{align*}
    Hence, we get
    \begin{align*}
        z^3 &= (-x -y)^3 = - \left( xyx + xy^2 + yxy + y^2x\right) , \\
        0 = z^4 &= (-x -y)^4 
        = (xy)^2 + xy^2x +(yx)^2 + yxy^2 + y^2xy
        .
    \end{align*}
    Then we derive the two equalities
    \begin{align}
        \tag{*}
        \label{eq:E7}
    y^2xyxz^3 
    = - y^2xyxy^2x
    = yxy^2xy^2x.
    \end{align}
    Indeed, we have
    \begin{align*}
        y^2xyxz^3 &= - y^2xyx\left( xyx + xy^2 + yxy + y^2x\right) 
        = - y^2xyxy^2x - y^2(xy)^2 
        \\&
        = - y^2xyxy^2x + y^2\left(xy^2x +(yx)^2 + yxy^2 + y^2xy\right)xy 
        = - y^2xyxy^2x
    \end{align*}
    and
    \begin{align*}
        y^2xyxy^2x &= - \left( (xy)^2 + xy^2x +(yx)^2 + yxy^2 \right) xy^2x
        = - (xy)^3yx - yxy^2xy^2x
        \\&
        = x\left((xy)^2 + xy^2x + yxy^2 + y^2xy\right)y^2x  - yxy^2xy^2x
        = - yxy^2xy^2x
        .
    \end{align*}
    
    We will show now that the algebras $A'$ and $A''$ are isomorphic.
    Let $\varphi : A'' \to A'$ be the homomorphishm of algebras
    defined on the arrows as follows:
    \begin{align*} 
        \varphi({a}_0) &= {a}_0 
        ,\\
        \varphi(a_1) &= a_1 
        + \theta_1 
        a_1 a_2 \bar{a}_0 a_0 \bar{a}_2 a_2 \bar{a}_0 a_0 (\bar{a}_2 a_2)^2 \bar{a}_0 a_0 \bar{a}_2
        ,\\
        \varphi(a_2) &= a_2 
        + (\theta_2 - \theta_1)
        \bar{a}_1 a_1 a_2 \bar{a}_0 a_0 \bar{a}_2 a_2 \bar{a}_0 a_0 (\bar{a}_2 a_2)^2 \bar{a}_0 a_0
        ,\\
        \varphi(a_3) &= a_3 
        + (\theta_4 - \theta_5 + \theta_6)
        (\bar{a}_2 a_2)^2 \bar{a}_0 a_0 \bar{a}_2 a_2 \bar{a}_0 a_0 a_{3} a_{4} a_{5} \bar{a}_{5} \bar{a}_{4}
        ,\\
        \varphi(a_4) &= a_4 
        + (\theta_5 - \theta_6)
        \bar{a}_{3} (\bar{a}_2 a_2)^2 \bar{a}_0 a_0 \bar{a}_2 a_2 \bar{a}_0 a_0 a_{3} a_{4} a_{5} \bar{a}_{5}
        ,\\
        \varphi(a_5) &= a_5 
        + \theta_6
        \bar{a}_{4} \bar{a}_{3} (\bar{a}_2 a_2)^2 \bar{a}_0 a_0 \bar{a}_2 a_2 \bar{a}_0 a_0 a_{3} a_{4} a_{5}
        ,\\
        \varphi(\bar{a}_0) &= \bar{a}_0 
        + \theta_0 
        (\bar{a}_2 a_2)^2 \bar{a}_0 a_0 \bar{a}_2 a_2 \bar{a}_0 a_0 (\bar{a}_2 a_2)^2 \bar{a}_0
        ,\\
        \varphi(\bar{a}_l) &= \bar{a}_l 
        , \mbox{ for } l \in \{ 1,\dots,5 \} 
        .
    \end{align*}
    We show that $\varphi$ is well defined.
    Indeed, we have the equalities
    \begin{align*}
        \varphi(a_0 \bar{a}_0) 
        &= \varphi(a_0) \varphi(\bar{a}_0) 
        \\&
        = a_0 \bar{a}_0 
        + \theta_0 
        a_0 (\bar{a}_2 a_2)^2 \bar{a}_0 a_0 \bar{a}_2 a_2 \bar{a}_0 a_0 (\bar{a}_2 a_2)^2 \bar{a}_0 
        = 0, 
       \\
%\end{align*}\begin{align*}
        \varphi(a_1 \bar{a}_1) 
        &= \varphi(a_1) \varphi(\bar{a}_1) 
        \\ &
        = a_1 \bar{a}_1 
        + \theta_1 
        a_1 a_2 \bar{a}_0 a_0 \bar{a}_2 a_2 \bar{a}_0 a_0 (\bar{a}_2 a_2)^2 \bar{a}_0 a_0 \bar{a}_2 \bar{a}_1
        = 0,
       \\
%\end{align*}\begin{align*}
        \varphi(\bar{a}_1 a_1 + a_2 \bar{a}_2) 
        &= \varphi(\bar{a}_1) \varphi(a_1) + \varphi(a_2) \varphi(\bar{a}_2) 
        \\ &
        = \bar{a}_1 a_1 + a_2 \bar{a}_2 
        + \theta_2 
        \bar{a}_1 a_1 a_2 \bar{a}_0 a_0 \bar{a}_2 a_2 \bar{a}_0 a_0 (\bar{a}_2 a_2)^2 \bar{a}_0 a_0 \bar{a}_2
        \\ &
        = 0,
       \\
%\end{align*}\begin{align*}
        \varphi(\bar{a}_3 a_3 + a_4 \bar{a}_4) 
        &= \varphi(\bar{a}_3) \varphi(a_3) + \varphi(a_4) \varphi(\bar{a}_4) 
        \\ &
        = \bar{a}_3 a_3 + a_4 \bar{a}_4 
        + \theta_4 
        \bar{a}_{3} (\bar{a}_2 a_2)^2 \bar{a}_0 a_0 \bar{a}_2 a_2 \bar{a}_0 a_0 a_{3} a_{4} a_{5} \bar{a}_{5} \bar{a}_4
        \\ &
        = 0,
        \\
%\end{align*}\begin{align*}
        \varphi(\bar{a}_4 a_4 + a_5 \bar{a}_5) 
        &= \varphi(\bar{a}_4) \varphi(a_4) + \varphi(a_5) \varphi(\bar{a}_5) 
        \\ &
        = \bar{a}_4 a_4 + a_5 \bar{a}_5 
        + \theta_5 
        \bar{a}_4 \bar{a}_{3} (\bar{a}_2 a_2)^2 \bar{a}_0 a_0 \bar{a}_2 a_2 \bar{a}_0 a_0 a_{3} a_{4} a_{5} \bar{a}_{5}
        \\ &
        = 0,
        \\
        \varphi(\bar{a}_{5} a_{5}) &= 
        \varphi(\bar{a}_{5}) \varphi(a_{5})
        \\ &= 
        \bar{a}_{5} a_{5}
        + \theta_{6} 
        \bar{a}_{5} \bar{a}_{4} \bar{a}_{3} (\bar{a}_2 a_2)^2 \bar{a}_0 a_0 \bar{a}_2 a_2 \bar{a}_0 a_0 a_{3} a_{4} a_{5}
        = 0
        .
    \end{align*}
Then, applying (\ref{eq:E7}), we obtain
    \begin{align*}
        \varphi\big(\bar{a}_{0} a_{0} + \bar{a}_{2} a_{2} + & a_{3} \bar{a}_{3}
        +
        \theta 
        \bar{a}_2 a_2 \bar{a}_0 a_0 (\bar{a}_2 a_2)^2 \bar{a}_0 a_0 (\bar{a}_2 a_2)^2 \bar{a}_0 a_0
        \big) 
        \\&=
        \varphi(\bar{a}_{0}) \varphi(a_{0}) + 
        \varphi(\bar{a}_{1}) \varphi(a_{1}) + \varphi(a_{2}) \varphi(\bar{a}_{2})
        \\ & \ \ \ 
        + \theta  
        \varphi(\bar{a}_2) \varphi(a_2) \left(\varphi(a_0) \varphi(\bar{a}_0) (\varphi(\bar{a}_{2}) \varphi(a_{2}))^2\right)^2  \varphi(\bar{a}_0) \varphi(a_0)
\end{align*}\begin{align*}
         &= 
        \bar{a}_{0} a_{0} + \bar{a}_{2} a_{2} + a_{3} \bar{a}_{3}
        \\ & \ \ \ 
        + \theta_0
        (\bar{a}_2 a_2)^2 \bar{a}_0 a_0 \bar{a}_2 a_2 \bar{a}_0 a_0 (\bar{a}_2 a_2)^2 \bar{a}_0 a_0
        \\ & \ \ \ 
        + (\theta_2 - \theta_1)
        \bar{a}_2 \bar{a}_1 a_1 a_2 \bar{a}_0 a_0 \bar{a}_2 a_2 \bar{a}_0 a_0 (\bar{a}_2 a_2)^2 \bar{a}_0 a_0
        \\ & \ \ \ 
        + (\theta_4 - \theta_5 + \theta_6)
        (\bar{a}_2 a_2)^2 \bar{a}_0 a_0 \bar{a}_2 a_2 \bar{a}_0 a_0 a_{3} a_{4} a_{5} \bar{a}_{5} \bar{a}_4 \bar{a}_3
        \\ & \ \ \ 
        + \theta  
        \bar{a}_2 a_2 \bar{a}_0 a_0 (\bar{a}_2 a_2)^2 \bar{a}_0 a_0 (\bar{a}_2 a_2)^2 \bar{a}_0 a_0
        \\
%\end{align*}\begin{align*}
     &= 
        \bar{a}_{0} a_{0} + \bar{a}_{2} a_{2} + a_{3} \bar{a}_{3}
        + \theta_3
        (\bar{a}_2 a_2)^2 \bar{a}_0 a_0 \bar{a}_2 a_2 \bar{a}_0 a_0 (\bar{a}_2 a_2)^2 \bar{a}_0 a_0
        \\ &= 0 ,
    \end{align*}
    so $\varphi$ is well defined.
    Similarly, we define the homomorphism of algebras
    $\psi : A' \to A''$ given on the arrows by
    \begin{align*} 
        \psi({a}_0) &= {a}_0 
        ,\\
        \psi(a_1) &= a_1 
        - \theta_1 
        a_1 a_2 \bar{a}_0 a_0 \bar{a}_2 a_2 \bar{a}_0 a_0 (\bar{a}_2 a_2)^2 \bar{a}_0 a_0 \bar{a}_2
        ,\\
        \psi(a_2) &= a_2 
        - (\theta_2 - \theta_1)
        \bar{a}_1 a_1 a_2 \bar{a}_0 a_0 \bar{a}_2 a_2 \bar{a}_0 a_0 (\bar{a}_2 a_2)^2 \bar{a}_0 a_0
        ,\\
        \psi(a_3) &= a_3 
        - (\theta_4 - \theta_5 + \theta_6)
        (\bar{a}_2 a_2)^2 \bar{a}_0 a_0 \bar{a}_2 a_2 \bar{a}_0 a_0 a_{3} a_{4} a_{5} \bar{a}_{5} \bar{a}_{4}
        ,\\
        \psi(a_4) &= a_4 
        - (\theta_5 - \theta_6)
        \bar{a}_{3} (\bar{a}_2 a_2)^2 \bar{a}_0 a_0 \bar{a}_2 a_2 \bar{a}_0 a_0 a_{3} a_{4} a_{5} \bar{a}_{5}
        ,
\\
%\end{align*}\begin{align*}
        \psi(a_5) &= a_5 
        - \theta_6
        \bar{a}_{4} \bar{a}_{3} (\bar{a}_2 a_2)^2 \bar{a}_0 a_0 \bar{a}_2 a_2 \bar{a}_0 a_0 a_{3} a_{4} a_{5}
        ,\\
        \psi(\bar{a}_0) &= \bar{a}_0 
        - \theta_0 
        (\bar{a}_2 a_2)^2 \bar{a}_0 a_0 \bar{a}_2 a_2 \bar{a}_0 a_0 (\bar{a}_2 a_2)^2 \bar{a}_0
        ,\\
        \psi(\bar{a}_l) &= \bar{a}_l 
        , \mbox{ for } l \in \{ 1,\dots,5 \} ,
    \end{align*}
    and prove that it is well defined.
    The compositions  $\psi \varphi$ and $\varphi \psi$
    are the identities on $A''$ and $A'$, respectively.
    Thus $\psi$ and $\varphi$ are isomorphisms,
    and the algebras $A''$ and $A'$ are isomorphic.
    Hence the algebras $A''$ and $A$ are also isomorphic.

    Therefore, it suffices to show that the algebra $A''$
    is isomorphic to one of the algebras 
    $P(\mathbb{E}_7)$, $P^{*}(\mathbb{E}_7)$.
    
    If $\theta = 0$, then $A''$ is, by the definition, 
         the preprojective algebra $P(\mathbb{E}_7)$.
    If $\theta \neq 0$, then there exists $\lambda \in K \setminus \{ 0 \}$ satisfying equation
    $\lambda^{13} = \theta$.
    Then it is  easily seen that
    the homomorphism 
    $\phi : P^{*}(\mathbb{E}_7) \to A''$ 
    of algebras
    given on the arrows by 
    \begin{gather*}
        \phi(a_k) = \lambda a_k
        , \, 
        \psi(\bar{a}_k) = \lambda \bar{a}_k 
        , \mbox{ for }  k  \in \{ 0,\dots,5 \} ,
    \end{gather*}
    is an isomorphism. 
\end{proof}

\begin{lemma}
    \label{lem:E8-1}
Let $A$ be a self-injective algebra socle equivalent to the preprojective algebra
$P(\mathbb{E}_8)$.
Then $A$ is isomorphic to one of the algebras 
$P(\mathbb{E}_8)$ or $P^{*}(\mathbb{E}_8)$.
\end{lemma}

\begin{proof}
The algebra $A$ is isomorphic with a bound quiver 
algebra $A'$ of the quiver $Q_{\mathbb{E}_8}$
    \[
    \xymatrix{
        &&0 \ar@<.5ex>^{a_0}[d] \\
        1 \ar@<.5ex>^{a_1}[r] & 2 \ar@<.5ex>^{\bar{a}_1}[l] \ar@<.5ex>^{a_2}[r] &
        3 \ar@<.5ex>^{\bar{a}_2}[l] \ar@<.5ex>^{a_3}[r] \ar@<.5ex>^{\bar{a}_0}[u] &
        4 \ar@<.5ex>^{\bar{a}_3}[l] \ar@<.5ex>^{a_4}[r] &
        5 \ar@<.5ex>^(.5){\bar{a}_4}[l] \ar@<.5ex>^{a_5}[r] &
        6 \ar@<.5ex>^(.5){\bar{a}_5}[l] \ar@<.5ex>^{a_6}[r] &
        7 \ar@<.5ex>^(.5){\bar{a}_6}[l] \\
    }
    \]
         bound by the relations:
    \begin{gather*}
        a_0 \bar{a}_0 + \theta_0 
        a_0 \left((\bar{a}_2 a_2)^2 \bar{a}_0 a_0\right)^2 \bar{a}_2 a_2 \left(\bar{a}_0 a_0 (\bar{a}_2 a_2)^2\right)^2 \bar{a}_0
        = 0 
        ,\\ 
        a_1 \bar{a}_1 + \theta_1 
        a_1 a_2 \bar{a}_0 a_0 \bar{a}_2 a_2 \left(\bar{a}_2 a_2 \bar{a}_0 a_0\right)^2 \left((\bar{a}_2 a_2)^2 \bar{a}_0 a_0\right)^2 \bar{a}_2 \bar{a}_1
        = 0 
        ,\\
        a_0 \bar{a}_0 a_0 = 0
        , \ \ 
        \bar{a}_0 a_0 \bar{a}_0 = 0
        , \ \ 
        a_1 \bar{a}_1 a_1 = 0
        , \ \ 
        \bar{a}_1 a_1 \bar{a}_1 = 0
        ,\\
        \bar{a}_0 a_0 + \bar{a}_2 a_2 + a_3 \bar{a}_3 + 
        \theta_3 
        \left((\bar{a}_2 a_2)^2 \bar{a}_0 a_0\right)^2 \bar{a}_2 a_2 \bar{a}_0 a_0 \left((\bar{a}_2 a_2)^2 \bar{a}_0 a_0\right)^2
        = 0
        ,\\
        a_0 (\bar{a}_2 a_2 + a_3 \bar{a}_3) = 0
        , \ \ 
        (\bar{a}_2 a_2 + a_3 \bar{a}_3) \bar{a}_0 = 0
        ,\\
        a_2 (\bar{a}_0 a_0 + \bar{a}_2 a_2 + a_3 \bar{a}_3) = 0
        , \ \ 
        (\bar{a}_0 a_0 + \bar{a}_2 a_2 + a_3 \bar{a}_3) \bar{a}_2 = 0
        ,\\
        \bar{a}_3 (\bar{a}_0 a_0 + \bar{a}_2 a_2 + a_3 \bar{a}_3) = 0
        , \ \ 
        (\bar{a}_0 a_0 + \bar{a}_2 a_2 + a_3 \bar{a}_3) a_3 = 0
        ,\\
        \bar{a}_{1} a_{1} +  a_{2} \bar{a}_{2}
        +
        \theta_2 
        a_2 \bar{a}_0 a_0 \bar{a}_2 a_2 \left(\bar{a}_2 a_2 \bar{a}_0 a_0\right)^2 \left((\bar{a}_2 a_2)^2 \bar{a}_0 a_0\right)^2 \bar{a}_2 \bar{a}_1 a_1
        = 0
        ,
\\
%\end{gather*}\begin{gather*}
        \bar{a}_{3} a_{3} +  a_{4} \bar{a}_{4}
        +
        \theta_4 
        \bar{a}_{3} 
        \left((\bar{a}_2 a_2)^2 \bar{a}_0 a_0\right)^2 \bar{a}_2 a_2 \bar{a}_0 a_0 (\bar{a}_2 a_2)^2
        a_{3} a_{4} a_{5} a_{6} \bar{a}_{6} \bar{a}_{5} \bar{a}_{4}
        = 0
        ,\\
        \bar{a}_{4} a_{4} +  a_{5} \bar{a}_{5}
        +
        \theta_5 
        \bar{a}_{4} \bar{a}_{3} 
        \left((\bar{a}_2 a_2)^2 \bar{a}_0 a_0\right)^2 \bar{a}_2 a_2 \bar{a}_0 a_0 (\bar{a}_2 a_2)^2
        a_{3} a_{4} a_{5} a_{6} \bar{a}_{6} \bar{a}_{5}
        = 0
        ,\\
        \bar{a}_{5} a_{5} +  a_{6} \bar{a}_{6}
        +
        \theta_6 
        \bar{a}_{5} \bar{a}_{4} \bar{a}_{3} 
        \left((\bar{a}_2 a_2)^2 \bar{a}_0 a_0\right)^2 \bar{a}_2 a_2 \bar{a}_0 a_0 (\bar{a}_2 a_2)^2
        a_{3} a_{4} a_{5} a_{6} \bar{a}_{6}
        = 0
        ,
\\
%\end{gather*}\begin{gather*}
        a_{k-1} (\bar{a}_{k-1} a_{k-1} + a_k \bar{a}_k) = 0
        , \ \ 
        (\bar{a}_{k-1} a_{k-1} + a_k \bar{a}_k) \bar{a}_{k-1} = 0
        ,\\
        \bar{a}_k (\bar{a}_{k-1} a_{k-1} + a_k \bar{a}_k) = 0
        , \ \ 
        (\bar{a}_{k-1} a_{k-1} + a_k \bar{a}_k) a_k = 0
        , \mbox{ for } k \in \{ 2,4,5,6 \}
        ,
\\
%\end{gather*}\begin{gather*}
        \bar{a}_{6} a_{6}
        +
        \theta_{7} 
        \bar{a}_{6} \bar{a}_{5} \bar{a}_{4} \bar{a}_{3} 
        \left((\bar{a}_2 a_2)^2 \bar{a}_0 a_0\right)^2 \bar{a}_2 a_2 \bar{a}_0 a_0 (\bar{a}_2 a_2)^2
        a_{3} a_{4} a_{5} a_{6}
        = 0
        ,\\
        \bar{a}_{6} a_{6} \bar{a}_{6} = 0
        , \ \ 
        a_{6} \bar{a}_{6} a_{6} = 0
        ,
    \end{gather*}
        for some coefficients 
        $\theta_0, \dots, \theta_{7} \in K$.
        Moreover, let $A''$ be the bound algebra of the 
        quiver $Q_{\mathbb{E}_8}$
        bound by the relations
    \begin{gather*}
        a_0 \bar{a}_0 = 0
        , \ \ 
        a_1 \bar{a}_1 = 0
        , \ \ 
        \bar{a}_{k-1} a_{k-1} + a_k \bar{a}_k = 0
        , \mbox{ for } k \in \{ 2,4,5,6 \}
        , \ \ 
        \bar{a}_6 a_6 = 0
        ,\\
        \bar{a}_0 a_0 + \bar{a}_2 a_2 + a_3 \bar{a}_3 + 
        \theta 
        \left((\bar{a}_2 a_2)^2 \bar{a}_0 a_0\right)^2 \bar{a}_0 a_0 \bar{a}_2 a_2 \left((\bar{a}_2 a_2)^2 \bar{a}_0 a_0\right)^2
        = 0
        ,\\
        a_0 (\bar{a}_2 a_2 + a_3 \bar{a}_3) = 0
        , \ \ 
        (\bar{a}_2 a_2 + a_3 \bar{a}_3) \bar{a}_0 = 0
        ,\\
        a_2 (\bar{a}_0 a_0 + \bar{a}_2 a_2 + a_3 \bar{a}_3) = 0
        , \ \ 
        (\bar{a}_0 a_0 + \bar{a}_2 a_2 + a_3 \bar{a}_3) \bar{a}_2 = 0
        ,\\
        \bar{a}_3 (\bar{a}_0 a_0 + \bar{a}_2 a_2 + a_3 \bar{a}_3) = 0
        , \ \ 
        (\bar{a}_0 a_0 + \bar{a}_2 a_2 + a_3 \bar{a}_3) a_3 = 0
        ,
    \end{gather*}
    with 
    %$\theta' = \theta_0 + \theta_1 - \theta_2 - \theta_3 + \theta_4 - \theta_5 + \theta_6 - \theta_7$. 
    $\theta = - \theta_0 - \theta_1 + \theta_2 + \theta_3 - \theta_4 + \theta_5 - \theta_6 + \theta_7$. 
    
    Similarly, as in the proof of Lemma~\ref{lem:E7-1},
    to simplify the notation we will denote
    $x = \bar{a}_0 a_0$, $y = \bar{a}_2 a_2$ and $z = a_3 \bar{a}_3$.
    Then, in the both algebras $A'$ and $A''$, we have
    the equalities  
    \begin{align*}
        x^2 &= (\bar{a}_0 a_0)^2 = \bar{a}_0 (a_0 \bar{a}_0) a_0 = 0, \\
        y^3 &= (\bar{a}_2 a_2)^3 = \bar{a}_2 \bar{a}_1 (a_1 \bar{a}_1) a_1 a_2 = 0, \\
        %  z^4 &= (a_3 \bar{a}_3)^4 
        %       = - a_3 a_4 a_5 (\bar{a}_5 a_5) \bar{a}_5 \bar{a}_4 \bar{a}_3 = 0 
        z^5 &= (a_3 \bar{a}_3)^4 
        = - a_3 a_4 a_5 a_6 (\bar{a}_6 a_6) \bar{a}_6 \bar{a}_5 \bar{a}_4 \bar{a}_3 = 0 
        .
    \end{align*}
    Therefore, we obtain the equalities
    \begin{align*}
        z^3 &= (-x -y)^3 = - \left( xyx + xy^2 + yxy + y^2x\right), \\
        z^4 &= (-x -y)^4 
        = (xy)^2 + xy^2x +(yx)^2 + yxy^2 + y^2xy, \\
        0 = z^5 &= (-x -y)^5 
        \\&
        = - \big( (xy)^2x + (xy)^2y + xy^2xy +(yx)^2y + yxy^2x + y^2xyx + y^2xy^2 \big)
        .
    \end{align*}
    We note that
    the last of the above equalities 
    will be intensively used in the
    further calculations of this proof.
    First, we use it to derive the equalities
    \begin{align*}
                \tag{1}
                \label{eq:E8-1}
    y^2 (xy)^2x y^2 
      &=  - y^2 \big( (xy)^2y + xy^2xy +(yx)^2y + yxy^2x + y^2xyx
\\& \qquad \ \quad
             + y^2xy^2 \big) y^2 
\\&
         = 0,
         \\
                \tag{2}
                \label{eq:E8-2}
    x y^2 x y^2 x 
      &=  - x \big( (xy)^2x + (xy)^2y + xy^2xy +(yx)^2y + yxy^2x 
\\& \qquad \ \quad
              + y^2xyx \big) x 
\\ &
         = - (xy)^3 x.
    \end{align*}
   Then, applying \eqref{eq:E8-1} and \eqref{eq:E8-2} we obtain 
    \begin{align*}
                \tag{3}
                \label{eq:E8-3}
    (xy^2xy)^2 x 
      &=  
    - (xy^2xy)  
    \big((xy)^2x + (xy)^2y +(yx)^2y + yxy^2x + y^2xyx 
\\& \, \quad 
    + y^2xy^2\big) x 
\\ 
&
     = - (xy^2)^2 (xy)^2 x - x y^2 (xy)^2x y^2 x 
     = (xy)^5 x,
    \end{align*}
and the dual equality 
%$(xyxy^2)^2 x = (xy)^5 x$. 
    \begin{align*}
                \tag{$3'$}
                \label{eq:E8-3p}
    (xyxy^2)^2 x 
      &= (xy)^5 x.
    \end{align*}
    Further, we derive from
     (\ref{eq:E8-2}) and (\ref{eq:E8-3p})
    the equalities
    \begin{align*}
                \tag{4}
                \label{eq:E8-4}
    (xy)^2y(xy)^3 
      &=  
    - (xy)^2yx  
    \big((xy)^2x + (xy)^2y + xy^2xy + yxy^2x + y^2xyx 
\\& \qquad \qquad \qquad
    + y^2xy^2\big) x 
\\ &
     = 
     - xyxy^2xy^2x yx - (xyxy^2)^2 x - x y (xy^2)^3 
\\ &
    = (xy)^5 x - (xy)^5 x
     - x y (xy^2)^3 
     = - x y (xy^2)^3 .
    \end{align*}
    We have also the equality
    \begin{align*}
                \tag{5}
                \label{eq:E8-5}
    xy^2 xy (xy^2)^2
   &= - x \big( (xy)^2x + (xy)^2y + xy^2xy +(yx)^2y + y^2xyx 
\\& \qquad \ \quad
   + y^2xy^2 \big) y^2 x y^2
\\ &
     = - x y (xy^2)^3 .
    \end{align*}
    Then, applying  (\ref{eq:E8-1}), (\ref{eq:E8-3}) and (\ref{eq:E8-5})   
   we obtain
    \begin{align*}
                \tag{6}
                \label{eq:E8-6}
    xy^2 (xy)^4
   &= - xy^2xyx \big( (xy)^2x + (xy)^2y + xy^2xy + yxy^2x + y^2xyx 
\\& \qquad \qquad \qquad
   + y^2xy^2 \big)
\\ &
     = - (xy^2xy)^2 x - x y^2 (xy)^2x y^2 x - xy^2 xy (xy^2)^2
\\ &
     = - (xy)^5 x + x y (xy^2)^3  .
    \end{align*}
  Further, from (\ref{eq:E8-4}) and (\ref{eq:E8-6})
  we derive the equalities
\begin{align*}
                \tag{7}
                \label{eq:E8-7}
    (xy)^3 y (xy)^2
   &= - x\big( (xy)^2x + (xy)^2y + xy^2xy + yxy^2x + y^2xyx 
\\& \qquad \ \quad 
   + y^2xy^2 \big)y(xy)^2
\\ &
     = - xy^2 (xy)^4 - (xy)^2y(xy)^3
%\\ &
     = (xy)^5 x - x y (xy^2)^3 
\\ &
     = x y (xy^2)^3 
     = (xy)^5 x .
\end{align*}
We claim now that we have the equality
        \begin{align}
%                \tag{*}
                \tag{8}
                \label{eq:E8}
    \left(y^2x\right)^2yx\left(y^2x\right)^2
    =
    \left(y^2x\right)^2yxy^2z^4
    .
        \end{align}
Indeed, applying (\ref{eq:E8-7}) we obtain 
\begin{align*}
%%% \big( (xy)^2x + (xy)^2y + xy^2xy +(yx)^2y + yxy^2x + y^2xyx + y^2xy^2 \big)
  (y^2x)^2&yxy^2(xy)^2
 \\ 
   &= -y^2xy^2\big( (xy)^2x + xy^2xy +(yx)^2y + yxy^2x + y^2xyx + y^2xy^2 \big) (xy)^2
\\ &
  = - (y^2x)^2y^2 (xy)^3
 \\ 
   &= y\big( (xy)^2x + (xy)^2y + xy^2xy +(yx)^2y + y^2xyx + y^2xy^2 \big) y^2(xy)^3
\\ &
  =  y (xy)^3 y (xy)^2
%\\ &
  =  y (xy)^5 x^2 y
  = 0 ,
\end{align*}
and hence
\begin{align*}
  (y^2x)^2yxy^2z^4
   &= (y^2x)^2yxy^2\big((xy)^2 + xy^2x +(yx)^2 + yxy^2 + y^2xy\big)^4  
\\
   &= (y^2x)^2yxy^2(xy)^2 + (y^2x)^2yxy^2xy^2x 
\\
   &= (y^2x)^2yx(y^2x)^2 .
\end{align*}
    
     We will show now that the algebras $A'$ and $A''$ are isomorphic.
         Let $\varphi : A'' \to A'$ be the algebra homomorphism 
         defined  on the arrows as follows:
    \begin{align*} 
        \varphi({a}_0) &= {a}_0 
        ,\\
        \varphi(a_1) &= a_1 
        + \theta_1 
        a_1 a_2 \bar{a}_0 a_0 \bar{a}_2 a_2 \left(\bar{a}_2 a_2 \bar{a}_0 a_0\right)^2 \left((\bar{a}_2 a_2)^2 \bar{a}_0 a_0\right)^2 \bar{a}_2
        ,\\
        \varphi(a_2) &= a_2 
        + (\theta_2 - \theta_1)
        \bar{a}_1 a_1 a_2 \bar{a}_0 a_0 \bar{a}_2 a_2 \left(\bar{a}_2 a_2 \bar{a}_0 a_0\right)^2 \left((\bar{a}_2 a_2)^2 \bar{a}_0 a_0\right)^2
        ,
    \end{align*}\begin{align*} 
        \varphi(a_3) &= a_3 
\\&\quad
        + (\theta_4 - \theta_5 + \theta_6 - \theta_7)
        \left((\bar{a}_2 a_2)^2 \bar{a}_0 a_0\right)^2 \bar{a}_2 a_2 \bar{a}_0 a_0 (\bar{a}_2 a_2)^2
        a_{3} a_{4} a_{5} a_{6} \bar{a}_{6} \bar{a}_{5} \bar{a}_{4}
        ,
         \\
%    \end{align*}\begin{align*} 
        \varphi(a_4) &= a_4 
        + (\theta_5 - \theta_6 + \theta_7)
        \bar{a}_{3} 
        \left((\bar{a}_2 a_2)^2 \bar{a}_0 a_0\right)^2 \bar{a}_2 a_2 \bar{a}_0 a_0 (\bar{a}_2 a_2)^2
        a_{3} a_{4} a_{5} a_{6} \bar{a}_{6} \bar{a}_{5}
        ,
         \\
%    \end{align*} \begin{align*} 
        \varphi(a_5) &= a_5 
        + (\theta_6 - \theta_7)
        \bar{a}_{4} \bar{a}_{3} 
        \left((\bar{a}_2 a_2)^2 \bar{a}_0 a_0\right)^2 \bar{a}_2 a_2 \bar{a}_0 a_0 (\bar{a}_2 a_2)^2
        a_{3} a_{4} a_{5} a_{6} \bar{a}_{6}
        ,\\
        \varphi(a_6) &= a_6 
        + \theta_7
        \bar{a}_{5} \bar{a}_{4} \bar{a}_{3} 
        \left((\bar{a}_2 a_2)^2 \bar{a}_0 a_0\right)^2 \bar{a}_2 a_2 \bar{a}_0 a_0 (\bar{a}_2 a_2)^2
        a_{3} a_{4} a_{5} a_{6}
        ,\\
        \varphi(\bar{a}_0) &= \bar{a}_0 
        + \theta_0 
        \left((\bar{a}_2 a_2)^2 \bar{a}_0 a_0\right)^2 \bar{a}_2 a_2 \left(\bar{a}_0 a_0 (\bar{a}_2 a_2)^2\right)^2 \bar{a}_0
        ,\\
        \varphi(\bar{a}_l) &= \bar{a}_l 
        , \mbox{ for } l \in \{ 1,\dots,6 \} 
        .
    \end{align*}
    We show that $\varphi$ is well defined.
         Indeed, we have the equalities
    \begin{align*}
        \varphi(a_0 \bar{a}_0) 
        &= \varphi(a_0) \varphi(\bar{a}_0) 
        \\&
        = a_0 \bar{a}_0 
        + \theta_0 
        a_0 \left((\bar{a}_2 a_2)^2 \bar{a}_0 a_0\right)^2 \bar{a}_2 a_2 \left(\bar{a}_0 a_0 (\bar{a}_2 a_2)^2\right)^2 \bar{a}_0
        = 0, 
        \\
        \varphi(a_1 \bar{a}_1) 
        &= \varphi(a_1) \varphi(\bar{a}_1) 
        \\ &
        = a_1 \bar{a}_1 
        + \theta_1 
        a_1 a_2 \bar{a}_0 a_0 \bar{a}_2 a_2 \left(\bar{a}_2 a_2 \bar{a}_0 a_0\right)^2 \left((\bar{a}_2 a_2)^2 \bar{a}_0 a_0\right)^2 \bar{a}_2 \bar{a}_1
%\\&\quad
        = 0,
        \\
        \varphi(\bar{a}_1 a_1 + a_2 \bar{a}_2) 
        &= \varphi(\bar{a}_1) \varphi(a_1) + \varphi(a_2) \varphi(\bar{a}_2) 
        \\ &
        = \bar{a}_1 a_1 + a_2 \bar{a}_2 
\\&\quad
        + \theta_2 
        \bar{a}_1 a_1 a_2 \bar{a}_0 a_0 \bar{a}_2 a_2 \left(\bar{a}_2 a_2 \bar{a}_0 a_0\right)^2 \left((\bar{a}_2 a_2)^2 \bar{a}_0 a_0\right)^2 \bar{a}_2
        \\ &
        = 0,
\\
%\end{align*}\begin{align*}
        \varphi(\bar{a}_3 a_3 + a_4 \bar{a}_4) 
        &= \varphi(\bar{a}_3) \varphi(a_3) + \varphi(a_4) \varphi(\bar{a}_4) 
        \\ &
        = \bar{a}_3 a_3 + a_4 \bar{a}_4 
        \\ & \ \ \ 
        + \theta_4 
        \bar{a}_{3} 
        \left((\bar{a}_2 a_2)^2 \bar{a}_0 a_0\right)^2 \bar{a}_2 a_2 \bar{a}_0 a_0 (\bar{a}_2 a_2)^2
        a_{3} a_{4} a_{5} a_{6} \bar{a}_{6} \bar{a}_{5} \bar{a}_{4}
        \\ &
        = 0,
%\end{align*}\begin{align*}
       \\
        \varphi(\bar{a}_4 a_4 + a_5 \bar{a}_5) 
        &= \varphi(\bar{a}_4) \varphi(a_4) + \varphi(a_5) \varphi(\bar{a}_5) 
        \\ &
        = \bar{a}_4 a_4 + a_5 \bar{a}_5 
        \\ & \ \ \ 
        + \theta_5 
        \bar{a}_{4} \bar{a}_{3} 
        \left((\bar{a}_2 a_2)^2 \bar{a}_0 a_0\right)^2 \bar{a}_2 a_2 \bar{a}_0 a_0 (\bar{a}_2 a_2)^2
        a_{3} a_{4} a_{5} a_{6} \bar{a}_{6} \bar{a}_{5}
        \\ &
        = 0,
%\end{align*}\begin{align*}
        \\
        \varphi(\bar{a}_{5} a_{5} + a_6 \bar{a}_6) 
        &= \varphi(\bar{a}_5) \varphi(a_5) + \varphi(a_6) \varphi(\bar{a}_6) 
        \\ &
        = \bar{a}_5 a_5 + a_6 \bar{a}_6 
        \\ & \ \ \ 
        + \theta_6 
        \bar{a}_{5} \bar{a}_{4} \bar{a}_{3} 
        \left((\bar{a}_2 a_2)^2 \bar{a}_0 a_0\right)^2 \bar{a}_2 a_2 \bar{a}_0 a_0 (\bar{a}_2 a_2)^2
        a_{3} a_{4} a_{5} a_{6} \bar{a}_{6}
        \\ &
        = 0,
%\end{align*}\begin{align*}
        \\
        \varphi(\bar{a}_{6} a_{6}) &= 
        \varphi(\bar{a}_{6}) \varphi(a_{6})
        \\ &= 
        \bar{a}_{6} a_{6}
        + \theta_{7} 
        \bar{a}_{6} \bar{a}_{5} \bar{a}_{4} \bar{a}_{3} 
        \left((\bar{a}_2 a_2)^2 \bar{a}_0 a_0\right)^2 \bar{a}_2 a_2 \bar{a}_0 a_0 (\bar{a}_2 a_2)^2
        a_{3} a_{4} a_{5} a_{6}         
        \\ & 
        = 0
        .
    \end{align*}
Then, applying (\ref{eq:E8}), we obtain the equalities
    \begin{align*}
        \varphi\Big(\bar{a}_{0} a_{0} + & \bar{a}_{2} a_{2} + a_{3} \bar{a}_{3}
        +
        \theta 
        \left((\bar{a}_2 a_2)^2 \bar{a}_0 a_0\right)^2 \bar{a}_0 a_0 \bar{a}_2 a_2 \left((\bar{a}_2 a_2)^2 \bar{a}_0 a_0\right)^2
        \Big) 
        \\&=
        \varphi(\bar{a}_{0}) \varphi(a_{0}) + 
        \varphi(\bar{a}_{1}) \varphi(a_{1}) + \varphi(a_{2}) \varphi(\bar{a}_{2})
        \\ & \ \ \ 
        + \theta  
        \varphi(\bar{a}_2) \varphi(a_2) 
        \dots
        \varphi(\bar{a}_0) \varphi(a_0)
%\end{align*}\begin{align*}
        \\
         &= 
        \bar{a}_{0} a_{0} + \bar{a}_{2} a_{2} + a_{3} \bar{a}_{3}
        \\ & \ \ \ 
        + \theta_0
        \left((\bar{a}_2 a_2)^2 \bar{a}_0 a_0\right)^2 \bar{a}_2 a_2 \left(\bar{a}_0 a_0 (\bar{a}_2 a_2)^2\right)^2 \bar{a}_0 a_0
        \\ & \ \ \ 
        + (\theta_2 - \theta_1)
        \bar{a}_2 \bar{a}_1 a_1 a_2 \bar{a}_0 a_0 \bar{a}_2 a_2 \left(\bar{a}_2 a_2 \bar{a}_0 a_0\right)^2 \left((\bar{a}_2 a_2)^2 \bar{a}_0 a_0\right)^2
        \\ & \ \ \ 
        + (\theta_4 - \theta_5 + \theta_6 - \theta_7)
        \\ & \qquad 
        \left((\bar{a}_2 a_2)^2 \bar{a}_0 a_0\right)^2 \bar{a}_2 a_2 \bar{a}_0 a_0 (\bar{a}_2 a_2)^2
        a_{3} a_{4} a_{5} a_{6} \bar{a}_{6} \bar{a}_{5} \bar{a}_{4} \bar{a}_3
        \\ & \ \ \ 
        + \theta  
        \left((\bar{a}_2 a_2)^2 \bar{a}_0 a_0\right)^2 \bar{a}_0 a_0 \bar{a}_2 a_2 \left((\bar{a}_2 a_2)^2 \bar{a}_0 a_0\right)^2
%\end{align*}\begin{align*}
        \\
        &= 
        \bar{a}_{0} a_{0} + \bar{a}_{2} a_{2} + a_{3} \bar{a}_{3}
        \\ & \ \ \ 
        + \theta_3
        \left((\bar{a}_2 a_2)^2 \bar{a}_0 a_0\right)^2 \bar{a}_0 a_0 \bar{a}_2 a_2 \left((\bar{a}_2 a_2)^2 \bar{a}_0 a_0\right)^2
        \\ &= 0 ,
    \end{align*}
        so $\varphi$ is well defined.
        Similarly, we define the algebra homomorphism
        $\psi : A' \to A''$ given on the arrows by
    \begin{align*} 
        \psi({a}_0) &= {a}_0 
        ,\\
        \psi(a_1) &= a_1 
        - \theta_1 
        a_1 a_2 \bar{a}_0 a_0 \bar{a}_2 a_2 \left(\bar{a}_2 a_2 \bar{a}_0 a_0\right)^2 \left((\bar{a}_2 a_2)^2 \bar{a}_0 a_0\right)^2 \bar{a}_2
        ,\\
        \psi(a_2) &= a_2 
        - (\theta_2 - \theta_1)
        \bar{a}_1 a_1 a_2 \bar{a}_0 a_0 \bar{a}_2 a_2 \left(\bar{a}_2 a_2 \bar{a}_0 a_0\right)^2 \left((\bar{a}_2 a_2)^2 \bar{a}_0 a_0\right)^2
        ,\\
        \psi(a_3) &= a_3 
        - (\theta_4 - \theta_5 + \theta_6 - \theta_7)
        \left((\bar{a}_2 a_2)^2 \bar{a}_0 a_0\right)^2 \bar{a}_2 a_2 \bar{a}_0 a_0 (\bar{a}_2 a_2)^2
        a_{3} a_{4} a_{5} a_{6} \bar{a}_{6} \bar{a}_{5} \bar{a}_{4}
        ,
%\end{align*}\begin{align*}
       \\
        \psi(a_4) &= a_4 
        - (\theta_5 - \theta_6 + \theta_7)
        \bar{a}_{3} 
        \left((\bar{a}_2 a_2)^2 \bar{a}_0 a_0\right)^2 \bar{a}_2 a_2 \bar{a}_0 a_0 (\bar{a}_2 a_2)^2
        a_{3} a_{4} a_{5} a_{6} \bar{a}_{6} \bar{a}_{5}
        ,
%\end{align*}\begin{align*}
        \\
        \psi(a_5) &= a_5 
        - (\theta_6 - \theta_7)
        \bar{a}_{4} \bar{a}_{3} 
        \left((\bar{a}_2 a_2)^2 \bar{a}_0 a_0\right)^2 \bar{a}_2 a_2 \bar{a}_0 a_0 (\bar{a}_2 a_2)^2
        a_{3} a_{4} a_{5} a_{6} \bar{a}_{6}
        ,\\
        \psi(a_6) &= a_6 
        - \theta_7
        \bar{a}_{5} \bar{a}_{4} \bar{a}_{3} 
        \left((\bar{a}_2 a_2)^2 \bar{a}_0 a_0\right)^2 \bar{a}_2 a_2 \bar{a}_0 a_0 (\bar{a}_2 a_2)^2
        a_{3} a_{4} a_{5} a_{6}
        ,\\
        \psi(\bar{a}_0) &= \bar{a}_0 
        = \theta_0 
        \left((\bar{a}_2 a_2)^2 \bar{a}_0 a_0\right)^2 \bar{a}_2 a_2 \left(\bar{a}_0 a_0 (\bar{a}_2 a_2)^2\right)^2 \bar{a}_0
        ,\\
        \psi(\bar{a}_l) &= \bar{a}_l 
        , \mbox{ for } l \in \{ 1,\dots,6 \} 
        ,
    \end{align*}
        and prove that it is well defined.
        The compositions  $\psi \varphi$ and $\varphi \psi$
        are the identities on algebras $A''$ and $A'$, respectively.
        Hence $\psi$ and $\varphi$ are isomorphisms,
        and the algebras $A''$ and $A'$ are isomorphic.
        Thus the algebras $A''$ and $A$ are also isomorphic.
    
        Therefore, it suffices to show that the algebra $A''$
        is isomorphic to the algebra $P(\mathbb{E}_8)$ 
        or is isomorphic to the algebra $P^{*}(\mathbb{E}_8)$.
    
        As before, we consider the cases.
        If $\theta = 0$, then $A''$ is, 
        by  definition,
        the preprojective algebra $P(\mathbb{E}_8)$.
        If $\theta \neq 0$, then there exists $\lambda \in K$ satisfying equation
        $\lambda^{27} = \theta$.
        It can be easily seen that
        the homomorphism of algebras
        $\phi : P^{*}(\mathbb{E}_8) \to A''$ 
        given on the arrows by
    \begin{gather*}
        \phi(a_k) = \lambda a_k
        , \, 
        \psi(\bar{a}_k) = \lambda \bar{a}_k 
        , \mbox{ for }  k \in \{ 0,\dots,6 \} 
        ,
    \end{gather*}
        is an isomorphism. 
        This finishes the proof. 
\end{proof}

\section{Deformed preprojective algebras in characteristic $2$}
\label{sec:char-2}

The aim of this section is to prove the following proposition.

\begin{prop}
\label{prop:char-2}
Algebras socle equivalent but
not isomorphic 
to preprojective algebras of generalized Dynkin type
exist only in characteristic $2$.
\end{prop}

We recall, that the number of 
``candidates'' for 
algebras socle equivalent but
not isomorphic 
to preprojective algebras of generalized Dynkin type
was reduced by 
Propositions \ref{prop:no-def} and \ref{prop:iso}.
So it is enough to consider 
only the algebras $P^*(\Delta)$
with
$\Delta \in \{ \mathbb{D}_{2m}, \mathbb{E}_7, \mathbb{E}_8, \mathbb{L}_{n} \}$, 
and $m,n \geq 2$.
We recall also that
the algebras $P^*(\bL_n)$ and $P(\bL_n)$ 
are non-isomorphic only in characteristic $2$.
Proposition~\ref{prop:char-2} 
will follow from three lemmas.

\begin{lem}
\label{eq:Dn}
Let $n \geq 4$ be an even integer and 
$K$ 
%a field 
of characteristic different from $2$.
Then the algebras $P(\mathbb{D}_n)$ and $P^*(\mathbb{D}_n)$ 
over $K$
are isomorphic.
\end{lem}

\begin{proof}
Let $\operatorname{char} K \neq 2$ and $n = 2m$ for some integer $m \geq 2$.
Recall that $P^*(\mathbb{D}_n)$ is the bound quiver algebra of the quiver $Q_{\mathbb{D}_n}$
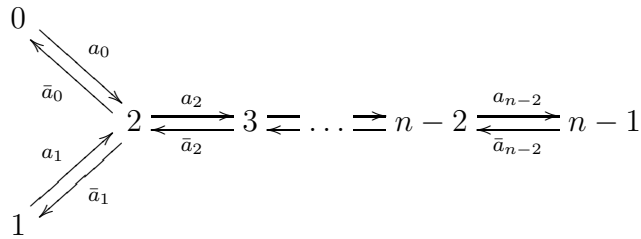
\[
    \xymatrix@C=1pc{
        0 \ar@<.5ex>^{a_0}[rrd] \\
         && 2 \ar@<.5ex>^{\bar{a}_0}[llu] \ar@<.5ex>^{\bar{a}_1}[lld] \ar@<.5ex>^{a_2}[rr] && 3
        \ar@<.5ex>^{\bar{a}_2}[ll]
         \ar@{-}@<.5ex>[r] & \ar@<.5ex>[l] \dots \ar@<.5ex>[r] & \ar@{-}@<.5ex>[l]
         n-2 \ar@<.5ex>^{a_{n-2}}[rr] &&
         n - 1 \ar@<.5ex>^(.5){\bar{a}_{n-2}}[ll] \\
        1 \ar@<.5ex>^{a_1}[rru] \\
    }
\]
bound by the relations:
\begin{gather*}
  a_0 \bar{a}_0 = 0 
  , \
  a_1 \bar{a}_1 = 0
  , \
  \bar{a}_0 a_0 + \bar{a}_1 a_1 +  a_2 \bar{a}_2
  +
 f(\bar{a}_0 a_0, \bar{a}_1 a_1)
   = 0
 ,\\
  \bar{a}_{k-1} a_{k-1} +  a_{k} \bar{a}_{k}
   = 0
  , \mbox{ for } k \in \{ 3, \dots, n-2 \}
  , \ \
  \bar{a}_{n-2} a_{n-2}
   = 0
,\\
  a_0 (\bar{a}_1 a_1 + a_2 \bar{a}_2) = 0
  , \ \
  a_1 (\bar{a}_0 a_0 + a_2 \bar{a}_2) = 0
  , \ \
  \bar{a}_2 (\bar{a}_0 a_0 + \bar{a}_1 a_1 + a_2 \bar{a}_2) = 0
 ,\\
  (\bar{a}_1 a_1 + a_2 \bar{a}_2) \bar{a}_0 = 0
  , \ \
  (\bar{a}_0 a_0 + a_2 \bar{a}_2) \bar{a}_1 = 0
  , \ \
  (\bar{a}_0 a_0 + \bar{a}_1 a_1 + a_2 \bar{a}_2) a_2 = 0
  ,
\end{gather*}
with $f(\bar{a}_0 a_0, \bar{a}_1 a_1) =  (\bar{a}_0 a_0 \bar{a}_1 a_1)^{m-1}$.
We will show that the algebras 
$P(\mathbb{D}_{n})$ and $P^*(\mathbb{D}_n)$ are isomorphic.

Observe that, in the both algebras
$P(\mathbb{D}_n)$ and $P^*(\mathbb{D}_n)$,
we have the equalities
\begin{align*}
 0 &= (-1)^{n-1} (-1)^{\frac{(n-3)(n-2)}{2}-1} a_2 \dots a_{n-2} (\bar{a}_{n-2} a_{n-2}) \bar{a}_{n-2} \dots \bar{a}_2
  \\&
    = (-1)^{n-1} (a_2 \bar{a}_2)^{n-1}
    = (\bar{a}_0 a_0 + \bar{a}_1 a_1)^{n-1}
  \\&
    = \bar{a}_0 a_0 (\bar{a}_0 a_0 + \bar{a}_1 a_1)^{n-2}
      +  \bar{a}_1 a_1 (\bar{a}_0 a_0 + \bar{a}_1 a_1)^{n-2}
  \\&
    = \bar{a}_0 a_0 \bar{a}_1 a_1 (\bar{a}_0 a_0 + \bar{a}_1 a_1)^{n-3}
      + \bar{a}_1 a_1 \bar{a}_0 a_0 (\bar{a}_0 a_0 + \bar{a}_1 a_1)^{n-3}
  \\&
    = \dots
    = (\bar{a}_0 a_0 \bar{a}_1 a_1)^{\frac{n}{2}-1}
      + (\bar{a}_1 a_1 \bar{a}_0 a_0)^{\frac{n}{2}-1}
  \\&
    = (\bar{a}_0 a_0 \bar{a}_1 a_1)^{m-1}
      + (\bar{a}_1 a_1 \bar{a}_0 a_0)^{m-1}
      .
\end{align*}
Hence, there exist the algebra homomorphism
$\varphi : P(\mathbb{D}_n) \to P^*(\mathbb{D}_n)$
given on the arrows by
\begin{gather*}
  \varphi(a_0) = a_0 + \tfrac{1}{2} a_0 (\bar{a}_1 a_1 \bar{a}_0 a_0)^{\frac{n-4}{2}} \bar{a}_1 a_1
  , \ \
  \varphi(\bar{a}_0) = \bar{a}_0 - \tfrac{1}{2} (\bar{a}_1 a_1 \bar{a}_0 a_0)^{\frac{n-4}{2}} \bar{a}_1 a_1 \bar{a}_0
  ,
\\
  \varphi(a_k) = a_k
  , \ \
  \varphi(\bar{a}_l) = \bar{a}_l
\ \ \mbox{ for } \ l = 1,\dots,n-2.
\end{gather*}
We claim that $\varphi$ is well defined.
Indeed, we have the equalities
\begin{align*}
    \varphi(a_0 \bar{a}_0)
    &= \varphi(a_0) \varphi(\bar{a}_0)
    = a_0 \bar{a}_0 +
      \left(\tfrac{1}{2} - \tfrac{1}{2}\right)
      a_0 (\bar{a}_1 a_1 \bar{a}_0 a_0)^{\frac{n-4}{2}} \bar{a}_1 a_1 \bar{a}_0
    \\ &= 0, 
%\end{align*}\begin{align*}
    \\
    \varphi(a_1 \bar{a}_1) &= \varphi(a_1) \varphi(\bar{a}_1) = a_1 \bar{a}_1 = 0, 
%\end{align*}\begin{align*}
    \\
    \varphi(\bar{a}_0 a_0 + \bar{a}_1 a_1 + a_2 \bar{a}_2) &=
    \varphi(\bar{a}_0) \varphi(a_0) + \varphi(\bar{a}_1) \varphi(a_1) + \varphi(a_2) \varphi(\bar{a}_2)
    \\ &=
    \bar{a}_0 a_0
     + \bar{a}_1 a_1 + a_2 \bar{a}_2
   \\ &\ \ \
    + \tfrac{1}{2} \bar{a}_0 a_0 (\bar{a}_1 a_1 \bar{a}_0 a_0)^{\frac{n-4}{2}}\bar{a}_1 a_1
    - \tfrac{1}{2} (\bar{a}_1 a_1 \bar{a}_0 a_0)^{\frac{n-4}{2}}\bar{a}_1 a_1 \bar{a}_0 a_0
    \\ &=
    \bar{a}_0 a_0 + \bar{a}_1 a_1 + a_2 \bar{a}_2
    + (\bar{a}_0 a_0 \bar{a}_1 a_1)^{\frac{n-2}{2}},
    \\
    \varphi(\bar{a}_{l} a_{l} + a_{l+1} \bar{a}_{l+1}) &=
    \varphi(\bar{a}_{l}) \varphi(a_{l}) + \varphi(a_{l+1}) \varphi(\bar{a}_{l+1})
    \\ &=
    \bar{a}_{l} a_{l} + a_{l+1} \bar{a}_{l+1} = 0, \ \ \mbox{ for } \  l = 2,\dots,n-3, \\
    \varphi(\bar{a}_{n-2} a_{n-2}) &= \varphi(\bar{a}_{n-2}) \varphi(a_{n-2}) = \bar{a}_{n-2} a_{n-2} = 0,
\end{align*}
so $\varphi$ is well defined.
Moreover, there exists the algebra homomorphism
$\psi : P^*(\mathbb{D}_n) \to P(\mathbb{D}_n)$
given on the arrows by
$$
  \psi(a_0) = a_0 - \tfrac{1}{2} a_0 (\bar{a}_1 a_1 \bar{a}_0 a_0)^{\frac{n-4}{2}} \bar{a}_1 a_1
  , \ \
  \psi(\bar{a}_0) = \bar{a}_0 + \tfrac{1}{2} (\bar{a}_1 a_1 \bar{a}_0 a_0)^{\frac{n-4}{2}} \bar{a}_1 a_1 \bar{a}_0
  ,
$$ $$
  \psi(a_k) = a_k
  , \ \
  \psi(\bar{a}_l) = \bar{a}_l
\ \ \mbox{ for } \ l = 1,\dots,n-2,
$$
which is the inverse of $\varphi$, and hence $\varphi$ is an isomorphism.
This ends the proof.
\end{proof}

\begin{lemma}
\label{lem:E7-2}
Let $K$ be 
%a field 
of characteristic different $2$.
Then the algebras $P(\mathbb{E}_7)$ and $P^*(\mathbb{E}_7)$ 
over $K$
are 
isomorphic.
\end{lemma}

\begin{proof}
To simplify the notation 
we abbreviate, 
as in the proof of Lemma~\ref{lem:E7-1},
$x = \bar{a}_0 a_0$, 
$y = \bar{a}_2 a_2$ 
and 
$z = a_3 \bar{a}_3$.
We recall (see the proof of Lemma~\ref{lem:E7-1}), that then
in the both algebras
$P(\mathbb{E}_7)$ and $P^{*}(\mathbb{E}_7)$
we have 
the equalities
\[
   x^2 = 0 = y^3
\]
and
\[
  0 = (x+y)^4 = x y^2 x + y x y x + x y x y + y x y^2 + y^2 x y .
\]
Then we derive the equalities
\begin{align*}
  y x y^2 x y^2 x &= - y x y^2 (y x y x + x y x y + y x y^2 + y^2 x y) =
  - y x y^2 x y x y \\
  &= y x (x y^2 x + y x y x + x y x y + y x y^2) x y = y x y x y^2 x y \\
  &= - (x y^2 x + x y x y + y x y^2 + y^2 x y) y^2 x y 
    = - x y^2 x y^2 x y 
  .
\end{align*}
Therefore, restoring the original notation, we obtain the equality
\begin{align}
\tag{*}
\label{e7:*}
   \bar{a}_2 a_2 \left(\bar{a}_0 a_0 (\bar{a}_2 a_2)^2\right)^2 \bar{a}_0 a_0
    = - \left(\bar{a}_0 a_0 (\bar{a}_2 a_2)^2\right)^2 \bar{a}_0 a_0 \bar{a}_2 a_2
    .
\end{align}

Assume now that $\operatorname{char} K \neq 2$.
Let 
$\varphi : P(\mathbb{E}_7) \to P^{*}(\mathbb{E}_7)$ 
be the algebra homomorphism given by
\begin{align*}
  \varphi(a_2) &= a_2 + \tfrac{1}{2} a_2 \left(\bar{a}_0 a_0 (\bar{a}_2 a_2)^2\right)^2 \bar{a}_0 a_0,
\\
  \varphi(\bar{a}_2) &= \bar{a}_2 - \tfrac{1}{2} \left(\bar{a}_0 a_0 (\bar{a}_2 a_2)^2\right)^2 \bar{a}_0 a_0 \bar{a}_2
  , 
\\
  \varphi(a_l) &= a_l 
  , \quad
  \varphi(\bar{a}_l) = \bar{a}_l 
  , \quad \mbox{ for } l \in \{ 0,1,3,4,5 \}
  .
\end{align*}
We prove that $\varphi$ is well defined.
Indeed, we have the equalities
\begin{align*}
    \varphi(a_l \bar{a}_l) 
    &= \varphi(a_l) \varphi(\bar{a}_l) = a_l \bar{a}_l = 0, 
   \quad \mbox{ for } l \in \{ 0,1 \}
    \\
    \varphi(\bar{a}_{1} a_{1} + a_{2} \bar{a}_{2}) &= 
    \varphi(\bar{a}_{1}) \varphi(a_{1}) + \varphi(a_{2}) \varphi(\bar{a}_{2})
    \\ &= 
    \bar{a}_{1} a_{1} + a_{2} \bar{a}_{2} + 
    \left(\tfrac{1}{2}-\tfrac{1}{2}\right) a_2 \left(\bar{a}_0 a_0 (\bar{a}_2 a_2)^2\right)^2 \bar{a}_0 a_0 \bar{a}_2
    \\ &= 
    \bar{a}_{1} a_{1} + a_{2} \bar{a}_{2} = 0  
    , 
    \\
    \varphi(\bar{a}_{3} a_{3} + a_{4} \bar{a}_{4}) &= 
    \varphi(\bar{a}_{3}) \varphi(a_{3}) + \varphi(a_{4}) \varphi(\bar{a}_{4})
    = \bar{a}_{3} a_{3} + a_{4} \bar{a}_{4} = 0, \\
    \varphi(\bar{a}_{4} a_{4} + a_{5} \bar{a}_{5}) &= 
    \varphi(\bar{a}_{4}) \varphi(a_{4}) + \varphi(a_{5}) \varphi(\bar{a}_{5})
    = \bar{a}_{4} a_{4} + a_{5} \bar{a}_{5} = 0, \\
    \varphi(\bar{a}_{5} a_{5}) &= 
    \varphi(\bar{a}_{5}) \varphi(a_{5})
    = \bar{a}_{5} a_{5} = 0 .
\end{align*}
Then, applying  (\ref{e7:*}), we obtain the equalities
\begin{align*}
    \varphi(\bar{a}_{0} a_{0} + \bar{a}_{2} a_{2} + a_{3} \bar{a}_{3}) &= 
    \varphi(\bar{a}_{0}) \varphi(a_{0}) + \varphi(\bar{a}_{2}) \varphi(a_{2}) + \varphi(a_{3}) \varphi(\bar{a}_{3})
    \\ &= 
    \bar{a}_{0} a_{0} + \bar{a}_{2} a_{2} + a_{3} \bar{a}_{3} 
    + \tfrac{1}{2} \bar{a}_2 a_2 \left(\bar{a}_0 a_0 (\bar{a}_2 a_2)^2\right)^2 \bar{a}_0 a_0
    \\  &\ \ \ 
    - \tfrac{1}{2} \left(\bar{a}_0 a_0 (\bar{a}_2 a_2)^2\right)^2 \bar{a}_0 a_0 \bar{a}_2 a_2
    \\ &\stackrel{\!(^*)\!}{=}
    \bar{a}_{0} a_{0} + \bar{a}_{2} a_{2} + a_{3} \bar{a}_{3} 
    +  \bar{a}_2 a_2 \left(\bar{a}_0 a_0 (\bar{a}_2 a_2)^2\right)^2 \bar{a}_0 a_0 
    \\ &= 0 , 
\end{align*}
so $\varphi$ is well defined.

It can be easily seen 
that the
algebra homomorphism
$\psi : P^{*}(\mathbb{E}_7) \to P(\mathbb{E}_7)$ 
given by
\begin{align*}
  \psi(a_2) &= a_2 - \tfrac{1}{2} a_2 \left(\bar{a}_0 a_0 (\bar{a}_2 a_2)^2\right)^2 \bar{a}_0 a_0,
\\
  \psi(\bar{a}_2) &= \bar{a}_2 + \tfrac{1}{2} \left(\bar{a}_0 a_0 (\bar{a}_2 a_2)^2\right)^2 \bar{a}_0 a_0 \bar{a}_2
  , 
\\
  \psi(a_l) &= a_l 
  , \quad
  \psi(\bar{a}_l) = \bar{a}_l 
  , \quad \mbox{ for } l \in \{ 0,1,3,4,5 \}
  ,
\end{align*}
is the inverse of $\varphi$, and hence $\varphi$ is an isomorphism.
This ends the proof.
\end{proof}

\begin{lemma}
\label{lem:E8-2}
Let $K$ be 
%a field 
of characteristic different $2$.
Then the algebras $P(\mathbb{E}_8)$ and $P^*(\mathbb{E}_8)$ 
over $K$ 
are isomorphic.
\end{lemma}

\begin{proof}
Again, as in the proof of Lemma~\ref{lem:E8-1},
to simplify the notation, 
we abbreviate
$x = \bar{a}_0 a_0$, $y = \bar{a}_2 a_2$ and $z = a_3 \bar{a}_3$.
Then, in the both algebras
$P(\mathbb{E}_7)$ and $P^{*}(\mathbb{E}_7)$,
we have
the equalities
(see the proof of Lemma~\ref{lem:E8-1})
\[
   x^2 = 0 = y^3 ,
\]
and
\[
   0 = - (x+y)^5 = (xy)^2x + (xy)^2y + xy^2xy +(yx)^2y + yxy^2x + y^2xyx + y^2xy^2 .
\]
From the above equalities, one may derive the equality
\[
  \left(y^2x\right)^2y\left(xy^2\right)^2x
  =
  x\left(y^2x\right)^2y\left(xy^2\right)^2
  .
\]
We note that the corresponding calculations are similar to the
calculations from the proofs of 
Lemmas \ref{lem:E7-1}, \ref{lem:E8-1}, and \ref{lem:E7-2},
but are considerably longer, so we omit the details.

Hence, after restoring the original notation, we obtain
\begin{align*}
\tag{*}
\label{e8:*}
  \left((\bar{a}_2 a_2)^2 \bar{a}_0 a_0\right)^2 \bar{a}_2 a_2 
 & 
  \left(\bar{a}_0 a_0 (\bar{a}_2 a_2)^2\right)^2  \bar{a}_0 a_0 
\\ &
  =
  \bar{a}_0 a_0 \left((\bar{a}_2 a_2)^2 \bar{a}_0 a_0\right)^2 \bar{a}_2 a_2 \left(\bar{a}_0 a_0 (\bar{a}_2 a_2)^2\right)^2 
 .
\end{align*}

Assume now that $\operatorname{char} K \neq 2$.
Let 
$\varphi : P(\mathbb{E}_8) \to P^{*}(\mathbb{E}_8)$ 
be the algebra homomorphism given by
\begin{align*}
  \varphi(a_0) &= a_0 - \tfrac{1}{2} a_0 \left((\bar{a}_2 a_2)^2 \bar{a}_0 a_0\right)^2 \bar{a}_2 a_2 \left(\bar{a}_0 a_0 (\bar{a}_2 a_2)^2\right)^2,
\\
  \varphi(\bar{a}_0) &= \bar{a}_0 + \tfrac{1}{2} \left((\bar{a}_2 a_2)^2 \bar{a}_0 a_0\right)^2 \bar{a}_2 a_2 \left(\bar{a}_0 a_0 (\bar{a}_2 a_2)^2\right)^2 \bar{a}_0
  , 
\\
  \varphi(a_l) &= a_l 
  , \quad
  \varphi(\bar{a}_l) = \bar{a}_l 
  , \quad \mbox{ for } l \in \{ 1,\dots, 6 \}
  .
\end{align*}
We claim that $\varphi$ is well defined.
Indeed, we have the equalities
\begin{align*}
    \varphi(a_0 \bar{a}_0) 
    &= \varphi(a_0) \varphi(\bar{a}_0) 
 \\&   
    = a_0 \bar{a}_0 
    + \left(\tfrac{1}{2}-\tfrac{1}{2}\right)
    a_0 \left((\bar{a}_2 a_2)^2 \bar{a}_0 a_0\right)^2 \bar{a}_2 a_2 \left(\bar{a}_0 a_0 (\bar{a}_2 a_2)^2\right)^2 \bar{a}_0
 \\&   
    = 0, 
    \\
    \varphi(a_1 \bar{a}_1) 
    &= \varphi(a_1) \varphi(\bar{a}_1) = a_1 \bar{a}_1 = 0, 
    \\
    \varphi(\bar{a}_{k-1} a_{k-1} + a_{k} \bar{a}_{k}) &= 
    \varphi(\bar{a}_{k-1}) \varphi(a_{k-1}) + \varphi(a_{k}) \varphi(\bar{a}_{k})
 \\&   
    = \bar{a}_{k-1} a_{k-1} + a_{k} \bar{a}_{k} 
    = 0, 
    \mbox{ for } k \in \{2,4,5,6\},
    \\
    \varphi(\bar{a}_{6} a_{6}) &= 
    \varphi(\bar{a}_{6}) \varphi(a_{6})
    = \bar{a}_{6} a_{6} = 0 ,
\end{align*}
and applying (\ref{e8:*}), we obtain
\begin{align*}
    \varphi(\bar{a}_{0} a_{0} + \bar{a}_{2} a_{2} + a_{3} \bar{a}_{3}) &= 
    \varphi(\bar{a}_{0}) \varphi(a_{0}) + \varphi(\bar{a}_{2}) \varphi(a_{2}) + \varphi(a_{3}) \varphi(\bar{a}_{3})
    \\ &= 
    \bar{a}_{0} a_{0} + \bar{a}_{2} a_{2} + a_{3} \bar{a}_{3} 
    \\  &\ \ \ 
    + \tfrac{1}{2} 
    \left((\bar{a}_2 a_2)^2 \bar{a}_0 a_0\right)^2 \bar{a}_2 a_2 \left(\bar{a}_0 a_0 (\bar{a}_2 a_2)^2\right)^2 \bar{a}_0 a_0
    \\  &\ \ \ 
    - t{1}{2} 
    \bar{a}_0 a_0 \left((\bar{a}_2 a_2)^2 \bar{a}_0 a_0\right)^2 \bar{a}_2 a_2 \left(\bar{a}_0 a_0 (\bar{a}_2 a_2)^2\right)^2
    \\ &\stackrel{\!(^*)\!}{=}
    \bar{a}_{0} a_{0} + \bar{a}_{2} a_{2} + a_{3} \bar{a}_{3} 
    \\  &\ \ \ 
    +  
    \left((\bar{a}_2 a_2)^2 \bar{a}_0 a_0\right)^2 \bar{a}_2 a_2 \left(\bar{a}_0 a_0 (\bar{a}_2 a_2)^2\right)^2 \bar{a}_0 a_0
    \\ &= 0 , 
\end{align*}
so $\varphi$ is well defined.
Similarly, one may prove that there exists an algebra homomorphism
$\psi : P^{*}(\mathbb{E}_8) \to P(\mathbb{E}_8)$ 
given by
\begin{align*}
  \psi(a_0) &= a_0 + \tfrac{1}{2} a_0 \left((\bar{a}_2 a_2)^2 \bar{a}_0 a_0\right)^2 \bar{a}_2 a_2 \left(\bar{a}_0 a_0 (\bar{a}_2 a_2)^2\right)^2,
\\
  \psi(\bar{a}_0) &= \bar{a}_0 - \tfrac{1}{2} \left((\bar{a}_2 a_2)^2 \bar{a}_0 a_0\right)^2 \bar{a}_2 a_2 \left(\bar{a}_0 a_0 (\bar{a}_2 a_2)^2\right)^2 \bar{a}_0
  , 
\\
  \psi(a_l) &= a_l 
  , \quad
  \psi(\bar{a}_l) = \bar{a}_l 
  , \quad \mbox{ for } l \in \{ 1,\dots, 6 \}
  ,
\end{align*}
which is the inverse of $\varphi$, and hence $\varphi$ is an isomorphism.
Therefore, the algebras $P(\mathbb{E}_8)$ and $P^{*}(\mathbb{E}_8)$ are isomorphic.
This ends the proof.
\end{proof}

\section{Symmetricity of socle deformed preprojective algebras}
\label{sec:non-sym}

The symmetricity of the preprojective algebras
of Dynkin types was studied by 
S.~Brenner, M.C.R.~Butler and S.~King in \cite{BBK}.
In particular, we have the following
consequence of \cite[Theorem~4.8]{BBK}.

\begin{thm}
\label{th:bbk}
The preprojective algebra $P(\Delta)$ of a Dynkin type $\Delta$
over $K$
different from $\mathbb{A}_1$
is symmetric if and only if 
$\charact K = 2$ 
and $\Delta$ 
is one of the types 
$\bD_{2n}$, $m \geq 2$, $\bE_{7}$, or $\bE_{8}$. 
\end{thm}

As a consequence of Theorem \ref{th:bbk}
we obtain the following 
lemma.

\begin{lem}
\label{lem:weaklysym}
Let 
$K$ be a field of characteristic $2$,
and $\Delta$ be one of the Dynkin types 
$\mathbb{D}_{2m}$, $m \geq 2$, $\mathbb{E}_7$, $\mathbb{E}_8$.
Then the algebra $P^*(\Delta)$
is weakly symmetric.
\end{lem}

\begin{proof}
Let 
$K$ be a field of characteristic $2$,
and $\Delta$ be one of the types 
$\mathbb{D}_{2m}$, $m \geq 2$, $\mathbb{E}_7$, $\mathbb{E}_8$.
Then, it follows from
Theorem~\ref{th:bbk} that
the algebra $P(\Delta)$ is symmetric,
and hence also weakly symmetric.
But then the algebra $P^*(\Delta)$
is also weakly symmetric.
\end{proof}

Our next aim is to
prove that
in characteristic $2$
the algebras 
$P^{*}(\bD_{2m})$, $m \geq 2$,
$P^{*}(\bE_7)$,
$P^{*}(\bE_8)$
are 
not symmetric.

\begin{lem}
\label{lem:nonsymDn}
Let $n \geq 4$ be an even integer and 
$K$ be
% a field 
 of characteristic $2$.
Then 
$P^*(\mathbb{D}_n)$
is 
%weakly symmetric but 
not symmetric.
\end{lem}

\begin{proof}
Recall that
$P^{*}(\mathbb{D}_n)$ 
is the bound quiver algebra of the quiver $Q_{\mathbb{D}_n}$
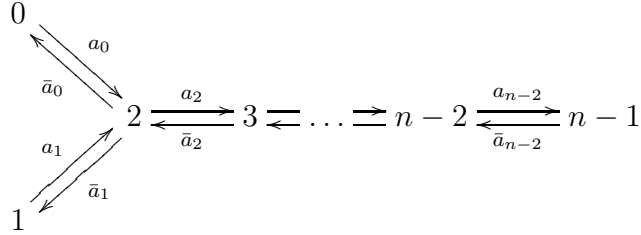
\[
    \xymatrix@C=1pc{
        0 \ar@<.5ex>^{a_0}[rrd] \\
        && 2 \ar@<.5ex>^{\bar{a}_0}[llu] \ar@<.5ex>^{\bar{a}_1}[lld] \ar@<.5ex>^{a_2}[rr] && 3
        \ar@<.5ex>^{\bar{a}_2}[ll]
        \ar@{-}@<.5ex>[r] & \ar@<.5ex>[l] \dots \ar@<.5ex>[r] & \ar@{-}@<.5ex>[l]
        n-2 \ar@<.5ex>^{a_{n-2}}[rr] &&
        n - 1 \ar@<.5ex>^(.5){\bar{a}_{n-2}}[ll] \\
        1 \ar@<.5ex>^{a_1}[rru] \\
    }
\]
bound by the relations:
    \begin{gather*}
        a_0 \bar{a}_0 = 0 
        , \
        a_1 \bar{a}_1 = 0
        , \
        \bar{a}_0 a_0 + \bar{a}_1 a_1 +  a_2 \bar{a}_2
        +
        f(\bar{a}_0 a_0, \bar{a}_1 a_1)
        = 0
        ,\\
        \bar{a}_{k-1} a_{k-1} +  a_{k} \bar{a}_{k}
        = 0
        , \mbox{ for } k \in \{ 3, \dots, n-2 \}
        , \ \
        \bar{a}_{n-2} a_{n-2}
        = 0
        ,\\
        a_0 (\bar{a}_1 a_1 + a_2 \bar{a}_2) = 0
        , \ \
        a_1 (\bar{a}_0 a_0 + a_2 \bar{a}_2) = 0
        , \ \
        \bar{a}_2 (\bar{a}_0 a_0 + \bar{a}_1 a_1 + a_2 \bar{a}_2) = 0
        ,\\
        (\bar{a}_1 a_1 + a_2 \bar{a}_2) \bar{a}_0 = 0
        , \ \
        (\bar{a}_0 a_0 + a_2 \bar{a}_2) \bar{a}_1 = 0
        , \ \
        (\bar{a}_0 a_0 + \bar{a}_1 a_1 + a_2 \bar{a}_2) a_2 = 0
          ,
    \end{gather*}
    with $f(\bar{a}_0 a_0, \bar{a}_1 a_1) = (\bar{a}_0 a_0 \bar{a}_1 a_1)^{\frac{n}{2}-1}$.

    Assume for a contradiction, that $P^{*}(\mathbb{D}_n)$ is symmetric.
    Then there exists a  symmetrizing form 
    $\psi : P^{*}(\mathbb{D}_n) \to K$
    (with $\Ker \psi$ without non-zero one-side ideals, 
    and $\psi(ab) = \psi(ba)$
    for all $a,b \in P^{*}(\mathbb{D}_n)$).
    Then $\psi\big(f(\bar{a}_0 a_0, \bar{a}_1 a_1)\big) \neq 0$, because
    $f(\bar{a}_0 a_0, \bar{a}_1 a_1)$ is
    an element from $\soc (P^{*}(\mathbb{D}_n))$.
    On the other hand, we have
    \begin{align*}
        \psi(\bar{a}_0 a_0) = \psi(a_0 \bar{a}_0) = \psi(0) = 0,
        \\
        \psi(\bar{a}_1 a_1) = \psi(a_1 \bar{a}_1 ) = \psi(0) = 0.
    \end{align*}
    Similarly we obtain
    \[
    \psi(a_2 \bar{a}_2 ) = \psi(\bar{a}_{n-1} a_{n-1}) =  \psi(0) = 0,
    \]
    because
    $\psi(a_i \bar{a}_i ) = \psi(\bar{a}_i a_i)$, for $i \in \{2,\dots,n-1\}$,
    and
    $\bar{a}_i a_i = a_{i+1} \bar{a}_{i+1}$, for $i \in \{2,\dots,n-2\}$.
    Hence
    \begin{align*}
        \psi\big(f(\bar{a}_0 a_0, \bar{a}_1 a_1)\big) &=
        \psi(\bar{a}_0 a_0)
        + \psi(\bar{a}_1 a_1)
        + \psi(a_2 \bar{a}_2)
        + \psi\big(f(\bar{a}_0 a_0, \bar{a}_1 a_1)\big)
        \\&=
        \psi\big(
        \bar{a}_0 a_0
        + \bar{a}_1 a_1
        + a_2 \bar{a}_2
        + f(\bar{a}_0 a_0, \bar{a}_1 a_1)
        \big)
        = \psi(0) = 0,
    \end{align*}
    a contradiction.
\end{proof}

\begin{lem}
\label{lem:nonsymEn}
Let $n \in \{ 7,8\}$ and 
$K$ be 
%a field 
of characteristic $2$.
Then $P^*(\mathbb{E}_n)$
is 
%weakly symmetric but 
not symmetric.
\end{lem}

\begin{proof}
Let $n \in \{ 7,8\}$.
    Recall that
    $P^{*}(\mathbb{E}_n) = P^{f}(\mathbb{E}_n)$
    is the bound quiver algebra of the quiver $Q_{\mathbb{E}_n}$
    \[
    \xymatrix{
        &&0 \ar@<.5ex>^{a_0}[d] \\
        1 \ar@<.5ex>^{a_1}[r] & 2 \ar@<.5ex>^{\bar{a}_1}[l] \ar@<.5ex>^{a_2}[r] &
        3 \ar@<.5ex>^{\bar{a}_2}[l] \ar@<.5ex>^{a_3}[r] \ar@<.5ex>^{\bar{a}_0}[u] &
        4 \ar@<.5ex>^{\bar{a}_3}[l] \ar@<.5ex>^{a_4}[r] &
        \dots \ar@<.5ex>^(.5){\bar{a}_4}[l] \ar@<.5ex>^{a_{n-2}}[r] &
        n-1 \ar@<.5ex>^(.5){\bar{a}_{n-2}}[l] \\
    }
    \]
    bound by the relations
    \begin{gather*}
        a_0 \bar{a}_0 = 0
        , \ \
        a_1 \bar{a}_1 = 0
        , \ \
        \bar{a}_0 a_0 + \bar{a}_2 a_2 + a_3 \bar{a}_3 +
        f(\bar{a}_0 a_0, \bar{a}_2 a_2)
        = 0
        ,\\
        \bar{a}_{k-1} a_{k-1} + a_k \bar{a}_k = 0
        , \mbox{ for } k \in \{ 2,4,\dots,n-2 \}
        , \ \
        \bar{a}_{n-2} a_{n-2} = 0
        ,\\
        a_0 (\bar{a}_2 a_2 + a_3 \bar{a}_3) = 0
        , \ \
        (\bar{a}_2 a_2 + a_3 \bar{a}_3) \bar{a}_0 = 0
        ,\\
        a_2 (\bar{a}_0 a_0 + \bar{a}_2 a_2 + a_3 \bar{a}_3) = 0
        , \ \
        (\bar{a}_0 a_0 + \bar{a}_2 a_2 + a_3 \bar{a}_3) \bar{a}_2 = 0
        ,\\
        \bar{a}_3 (\bar{a}_0 a_0 + \bar{a}_2 a_2 + a_3 \bar{a}_3) = 0
        , \ \
        (\bar{a}_0 a_0 + \bar{a}_2 a_2 + a_3 \bar{a}_3) a_3 = 0
          ,
    \end{gather*}
    with $f(\bar{a}_0 a_0, \bar{a}_1 a_1) = (\bar{a}_0 a_0 \bar{a}_1 a_1)^{3n-17}$.

    Assume that $P^{*}(\mathbb{E}_n)$ is symmetric.
    Then there exists a symmetrizing form $\psi : P^{*}(\mathbb{E}_n) \to K$.
    Hence $\psi\big(f(\bar{a}_0 a_0, \bar{a}_2 a_2)\big) \neq 0$, because
    $f(\bar{a}_0 a_0, \bar{a}_1 a_1)$ is
    an element from $\soc (P^{*}(\mathbb{E}_n))$.
    On the other hand
    \begin{align*}
        \psi(\bar{a}_0 a_0) = \psi(a_0 \bar{a}_0) = \psi(0) = 0,
        \\
        \psi(\bar{a}_2 a_2) = \psi(a_2 \bar{a}_2 ) =
        \psi(\bar{a}_1 a_1) = \psi(a_1 \bar{a}_1 ) = \psi(0) = 0.
    \end{align*}
    Similarly, we have
    \[
    \psi(a_3 \bar{a}_3 ) = \psi(\bar{a}_{n-2} a_{n-2}) =  \psi(0) = 0,
    \]
    because
    $\bar{a}_i a_i = a_{i+1} \bar{a}_{i+1}$, for $i \in \{3,\dots,n-3\}$,
    and
    $\psi(a_i \bar{a}_i ) = \psi(\bar{a}_i a_i)$, for $i \in \{3,\dots,n-2\}$.
    Hence
    \begin{align*}
        \psi\big(f(\bar{a}_0 a_0, \bar{a}_2 a_2)\big) &=
        \psi(\bar{a}_0 a_0)
        + \psi(\bar{a}_2 a_2)
        + \psi(a_3 \bar{a}_3)
        + \psi\big(f(\bar{a}_0 a_0, \bar{a}_2 a_2)\big)
        \\&=
        \psi\big(
        \bar{a}_0 a_0
        + \bar{a}_2 a_2
        + a_3 \bar{a}_3
        + f(\bar{a}_0 a_0, \bar{a}_2 a_2)
        \big)
        = \psi(0) = 0,
    \end{align*}
    a contradiction.
Therefore the algebra $P^*(\mathbb{E}_n)$
is not symmetric.   
\end{proof}

As a consequence of 
Lemmas 
\ref{lem:weaklysym},
\ref{lem:nonsymDn} and \ref{lem:nonsymEn}
we obtain the following theorem.

\begin{thm}
\label{th:non-sym}
Let $K$ be 
of characteristic $2$.
Then the algebras 
$P^{*}(\bD_{2m})$, $m \geq 2$,
$P^{*}(\bE_7)$,
$P^{*}(\bE_8)$
are 
weakly symmetric but
not symmetric.
\end{thm}

As an immediate consequence of 
Propositions \ref{prop:Ln-clas}, \ref{prop:char-2} and
Theorems \ref{thm:Ln-clas}, \ref{th:bbk}, \ref{th:non-sym}
we obtain also the following result.

\begin{thm}
\label{th:non-deriv}
Let
$\Delta \in \{ \mathbb{D}_{2m}, \mathbb{E}_7, \mathbb{E}_8, \mathbb{L}_{n} \}$, 
$m,n \geq 2$, 
and let
$P(\Delta)$ and $P^*(\Delta)$ be
the associated 
algebras over an algebraically closed field $K$.
Then the following statements are equivalent:
\begin{enumerate}[(i)]
\item
  $P^{*}(\Delta)$ and $P(\Delta)$
  are isomorphic. 
\item
  $K$ is not of characteristic $2$.
\end{enumerate}
\end{thm}

\begin{proof}
(i) $\Rightarrow$ (ii) \ 
Assume that  $K$ is of characteristic $2$.
If $\Delta$ is of type $\mathbb{L}_{n}$, $n \geq 2$,
then the implication follows from
Proposition~\ref{prop:Ln-clas}.
So assume 
that $\Delta$ is 
of Dynkin type.
Then it follows from
Theorem~\ref{th:bbk} that $P(\Delta)$ is symmetric,
and 
from
Theorem~\ref{th:non-sym} that $P^{*}(\Delta)$ is not symmetric.
Hence
$P^{*}(\Delta)$ and $P(\Delta)$
are not isomorphic. 

The implication (ii) $\Rightarrow$ (i)
is a consequence of 
Proposition \ref{prop:char-2}. 
\end{proof}

As a consequence of 
Propositions 
\ref{prop:no-def},
\ref{prop:iso},
\ref{prop:char-2},
and the above theorem
we obtain also the following 
characterization of 
socle deformed preprojective algebras of generalized Dynkin type.

\begin{cor}
\label{cor:full-family}
Socle deformed preprojective algebras of generalized Dynkin type 
exist only in characteristic $2$.
Moreover, in 
characteristic $2$,
the algebras
  $P^{*}(\bD_{2m})$, $m \geq 2$,
  $P^{*}(\bE_7)$,
  $P^{*}(\bE_8)$,
  $P^{*}(\bL_n)$, $n \geq 2$
form a complete family of 
pairwise non-isomorphic
socle deformed preprojective algebras of generalized Dynkin type.
\end{cor}

\section{Proofs of the main results}
\label{sec:last}

\begin{proof}[Proof of Theorem~\ref{thm1}]
Assume that $\Lambda$ is a self-injective algebra 
over an algebraically closed field $K$ which is
socle equivalent but not isomorphic
to a preprojective algebra $P(\Delta)$
of generalized Dynkin $\Delta$.
Then, by
Proposition~\ref{prop:char-2},
$K$ is of characteristic $2$.
Moreover, it follows 
Proposition~\ref{prop:no-def}, 
that
$\Delta$ has to be of one of the types
$\bD_{2m} (m \geq 2)$, $\bE_7$, $\bE_8$,
$\bL_{n} (n \geq 2)$. 
Hence, applying
Proposition~\ref{prop:iso},
we conclude that $\Lambda$ is isomorphic to one
of the algebras
  $P^{*}(\bD_{2m})$, $m \geq 2$,
  $P^{*}(\bE_7)$,
  $P^{*}(\bE_8)$,
  $P^{*}(\bL_n)$, $n \geq 2$.

Assume now that $K$ is of
characteristic $2$ and $\Lambda$ is isomorphic
to the algebra $P^*(\Delta)$
with 
$\Delta \in \{ \mathbb{D}_{2m}, \mathbb{E}_7, \mathbb{E}_8 , \mathbb{L}_{n} \}$, $m,n \geq 2$.
Then, by  definition of $P^*(\Delta)$, 
$\Lambda$ is socle equivalent to
the preprojective algebra $P(\Delta)$, 
and by Theorem~\ref{th:non-deriv},
$\Lambda$ is not isomorphic to $P(\Delta)$. 
\end{proof}

\begin{proof}[Proof of Theorem~\ref{thm3}]
Let $\Lambda$ be a socle deformed preprojective algebra
of generalized Dynkin type over an algebraically closed field $K$. 
Then, by Theorem~\ref{thm1},
$K$ is of characteristic $2$ and $\Lambda$ is isomorphic to one
of the algebras
$P^{*}(\bD_{2m})$, $m \geq 2$,
$P^{*}(\bE_7)$,
$P^{*}(\bE_8)$,
$P^{*}(\bL_n)$, $n \geq 2$.

(i)
If $\Lambda$ is isomorphic to the algebra
$P^{*}(\bL_n)$ for some $n \geq 2$,
then 
it follows from 
Proposition~\ref{thm:Ln-sym}
that 
$\Lambda$ is symmetric, and hence weakly symmetric.
In the other case (if $\Lambda$ is isomorphic to one
of the algebras
$P^{*}(\bD_{2m})$, $m \geq 2$,
$P^{*}(\bE_7)$,
$P^{*}(\bE_8)$),
$\Lambda$ is weakly symmetric
by Lemma~\ref{lem:weaklysym}.

(ii)
If $\Lambda$ is isomorphic to the algebra
$P^{*}(\bL_n)$ for some $n \geq 2$,
then $\Lambda$ is symmetric
by Proposition~\ref{thm:Ln-sym}.
On the other hand, if 
$\Lambda$ is not isomorphic to the algebra
of the form $P^{*}(\bL_n)$ for some $n \geq 2$, 
then $\Lambda$ is isomorphic to one
of the algebras
$P^{*}(\bD_{2m})$, $m \geq 2$,
$P^{*}(\bE_7)$,
$P^{*}(\bE_8)$,
and hence the thesis is a consequence
of Lemmas \ref{lem:nonsymDn} and \ref{lem:nonsymEn}.
\end{proof}

\begin{proof}[Proof of Theorem~\ref{thm4}]
Let $\Delta$ be a generalized Dynkin type different from $\bA_1$ and $\bL_1$,
and $K$ be an algebraically closed field.

The equivalence of statements (i) and (iii)
follows from Theorem~\ref{thm1}.

We claim that the statements (ii) and (iii)
are also equivalent.
Indeed, if $\Delta$ is a Dynkin type (so different from $\bL_n$),
then the equivalence of (ii) and (iii)
is a direct consequence of  
Theorem~\ref{th:bbk}.
On the other hand, if $\Delta = \bL_n$ for some 
$n \geq 2$, then each of the statements (ii) and (iii) is 
satisfied if and only if $K$ is of characteristic $2$.
\end{proof}

\subsection*{Acknowledgements}
This research was supported by
the Research Grant
DEC-2011/02/A/ST1/00216 of the National Science Center Poland.

\end{document}